\documentclass[11pt, reqno]{amsart}
\usepackage{txfonts}
\usepackage{amsmath,amssymb,amsthm,mathrsfs,enumerate,bm,xcolor,multirow,pbox}
\usepackage{algorithm,algorithmic,float}
\usepackage{graphicx,color,framed,tikz,caption,subcaption}
\usepackage{enumitem}
\setlist{leftmargin=6mm}
\usepackage[colorlinks,linkcolor=black,citecolor=black,urlcolor=black]{hyperref}
\allowdisplaybreaks[4]
\numberwithin{equation}{section}
\newcommand{\N}{\mathbb{N}}
\newcommand{\R}{\mathbb{R}}

\newcommand{\pnorm}[2]{\lVert #1\rVert_{#2}}
\newcommand{\bigpnorm}[2]{\big\lVert#1\big\rVert_{#2}}

\newcommand{\abs}[1]{\lvert#1\rvert}
\newcommand{\bigabs}[1]{\big\lvert#1\big\rvert}
\newcommand{\biggabs}[1]{\bigg\lvert#1\bigg\rvert}
\newcommand{\iprod}[2]{\langle#1,#2\rangle}
\newcommand{\bigiprod}[2]{\big\langle#1,#2\big\rangle}

\renewcommand{\epsilon}{\varepsilon}

\renewcommand{\d}[1]{\mathrm{d}#1}

\newcommand{\ceil}[1]{\left\lceil #1 \right\rceil}

\newcommand{\bigop}{\mathcal{O}_{\mathbf{P}}}

\newcommand{\smallo}{\mathfrak{o}}
\newcommand{\bigo}{\mathcal{O}}

\DeclareMathOperator{\E}{\mathbb{E}}
\DeclareMathOperator{\Prob}{\mathbb{P}}

\DeclareMathOperator{\op}{op}

\DeclareMathOperator{\err}{\texttt{err}}

\let\liminf\relax
\DeclareMathOperator*\liminf{\underline{lim}}

\DeclareMathOperator*{\argmax}{arg\,max\,}
\DeclareMathOperator*{\argmin}{arg\,min\,}

\newcommand{\beq}{\begin{equation}}
\newcommand{\eeq}{\end{equation}}
\newcommand{\beqa}{\begin{equation} \begin{aligned}}
\newcommand{\eeqa}{\end{aligned} \end{equation}}
\newcommand{\beqas}{\begin{equation*} \begin{aligned}}
\newcommand{\eeqas}{\end{aligned} \end{equation*}}

\newcommand{\bit}{\begin{itemize}}
	\newcommand{\eit}{\end{itemize}}
\newcommand{\bmat}{\begin{bmatrix}}
	\newcommand{\emat}{\end{bmatrix}}

\theoremstyle{definition}\newtheorem{problem}{Problem}[section]
\theoremstyle{definition}\newtheorem{definition}[problem]{Definition}
\theoremstyle{remark}\newtheorem{assumption}{Assumption}

\theoremstyle{remark}\newtheorem{remark}{Remark}
\theoremstyle{definition}
\theoremstyle{plain}\newtheorem{theorem}[problem]{Theorem}
\theoremstyle{plain}
\theoremstyle{plain}\newtheorem{lemma}[problem]{Lemma}
\theoremstyle{plain}\newtheorem{proposition}[problem]{Proposition}
\theoremstyle{plain}\newtheorem{corollary}[problem]{Corollary}
\theoremstyle{plain}

\AtBeginDocument{
	\def\MR#1{}
}

\begin{document}

\title[Regret and inference using UCB algorithms]{UCB algorithms for multi-armed bandits: \\ Precise regret and adaptive inference}

\author[Q. Han]{Qiyang Han}

\address[Q. Han]{
Department of Statistics, Rutgers University, Piscataway, NJ 08854, USA.
}
\email{qh85@stat.rutgers.edu}

\author[K. Khamaru]{Koulik Khamaru}

\address[K. Khamaru]{
	Department of Statistics, Rutgers University, Piscataway, NJ 08854, USA.
}
\email{kk1241@stat.rutgers.edu}

\author[C.-H. Zhang]{Cun-Hui Zhang}

\address[C.-H. Zhang]{
	Department of Statistics, Rutgers University, Piscataway, NJ 08854, USA.
}
\email{czhang@stat.rutgers.edu}

\date{\today}
\thanks{Authors listed in alphabetical order.}


\begin{abstract}
Upper Confidence Bound (UCB) algorithms are a widely-used class of sequential algorithms for the $K$-armed bandit problem. Despite extensive research over the past decades aimed at understanding their asymptotic and (near) minimax optimality properties, a precise understanding of their regret behavior remains elusive. This gap has not only hindered the evaluation of their actual algorithmic efficiency, but also limited further developments in statistical inference in sequential data collection.

This paper bridges these two fundamental aspects—precise regret analysis and adaptive statistical inference—through a deterministic characterization of the number of arm pulls for an UCB index algorithm 
\cite{lai1987adaptive, agrawal1995sample, auer2002finite}. Our resulting precise regret formula not only accurately captures the actual behavior of the UCB algorithm for finite time horizons and individual problem instances, but also provides significant new insights into the regimes in which the existing theoretical
frameworks remain informative.  In particular, we show that the classical Lai-Robbins regret
formula is exact if and only if the sub-optimality gaps exceed the
order $\sigma\sqrt{K\log T/T}$. We also show that its maximal regret deviates from the minimax regret by a logarithmic factor, and therefore settling its strict minimax optimality in the negative.

The deterministic characterization of the number of arm pulls for the UCB algorithm also has major implications in adaptive statistical inference. Building on the seminal work of \cite{lai1982least}, we show that the UCB algorithm satisfies certain `stability' properties that lead to quantitative central limit theorems in two settings: for the empirical means of unknown rewards in the bandit setting, and for a class of Ridge estimators when the arm means exhibit a structured relationship through covariates. These results have an important practical implication: conventional confidence sets designed for i.i.d. data remain valid even when data are collected sequentially.

Our technical approach relies on an application of a new comparison principle between the UCB algorithm and its noiseless, continuous-time minimax counterpart. We expect this new principle to be broadly applicable for general UCB index algorithms.
\end{abstract}

\maketitle

\newpage

\setcounter{tocdepth}{1}
\tableofcontents

\sloppy

\section{Introduction}
\subsection{Overview and the motivating questions}

Consider the standard $K$-armed bandit problem \cite{thompson1933likelihood,robbins1952some}: at each round $t\in [T]$, the player chooses an arm $A_t \in [K]$ based on the past rewards $R_1,\ldots,R_{t-1}$, and receives a new reward
\begin{align}
\label{eqn:bandit-model}
R_t\equiv \mu_{A_t}+\sigma\cdot  \xi_t,
\end{align}
where $\mu_a$ is the mean reward for the arm $a \in [K]$, and $\{\xi_t\}_{t \in [T]}$ are i.i.d. $\mathcal{N}(0,1)$ noises.
The optimal expected reward is denoted by $\mu_\ast = \max_{b \in [K]} \mu_b$, and $\Delta_a = \mu_\ast - \mu_a \geq 0$ represents the sub-optimality gap for arm $a \in [K]$.

The multi-armed bandit problem is among the most basic form for modern sequential decision-making problems, and has gained renewed interest due to its natural connections to modern complex reinforcement learning algorithms. For a comprehensive overview of the history of this problem and its wide-ranging applications, readers are referred to recent textbooks \cite{bubeck2012regret,slivkins2019introduction,lattimore2020bandit} and the references therein.

The goal for the bandit problem is to design an efficient algorithm $\mathscr{A}$ that selects the arms $\{A_t\}_{t \in [K]}$ to minimize the \emph{(expected) regret} over the $T$ rounds of the game, defined as
\begin{align}\label{def:regret_generic}
\texttt{Reg}(\mathscr{A})\equiv \E\sum_{t \in [T]} \big(\mu_\ast- R_t\big)=\E\sum_{t \in [T]} \Delta_{A_t}=\sum_{a: \Delta_a>0} \Delta_a \E n_{a;T}.
\end{align}
Here $n_{a;t}\equiv \sum_{s \in [t]} \bm{1}_{A_s =a}$ is the number of times arm $a$ has been pulled up to round $t \in [T]$. 

In this paper, we examine a widely-used class of Upper Confidence Bound (UCB) algorithms through the lens of two fundamental, yet seemingly unrelated, questions:
\begin{enumerate}
    \item[(Q1)] Can we characterize the \emph{precise regret} of the UCB algorithm for the multi-armed bandit problem~\eqref{eqn:bandit-model}?
    \item[(Q2)] Can we perform \emph{adaptive inference} for the mean rewards $\{\mu_a\}_{a \in [K]}$ based on data collected sequentially using the UCB algorithm?
\end{enumerate}
The main goal of this paper is to provide unified, affirmative answers to both (Q1) and (Q2) through a precise characterization of the numbers of arm pulls $\{n_{a;T}\}_{a \in [K]}$. While the connection between the number of arm pulls and the precise regret in (Q1) follows naturally from the very definition (\ref{def:regret_generic}), the relevance of this characterization to (Q2) appears less immediate. Interestingly, as we shall demonstrate, a deterministic characterization of the arm pulls is intrinsically linked to a certain `\emph{stability}' property of UCB algorithms, which plays a pivotal role—similar to the role in \cite{lai1982least}—in enabling statistical inference for $\{\mu_a\}_{a \in [K]}$.

\subsection{Precise regret}

Understanding the behavior of the regret $\texttt{Reg}(\mathscr{A})$ of a given bandit algorithm $\mathscr{A}$ lies at the heart of evaluating its algorithmic efficacy. To date, there are primarily two distinct but inter-related theoretical paradigms for regrets under which the optimality of a bandit algorithm $\mathscr{A}$ is evaluated:
\begin{enumerate}
	\item[(R1)]
    $\mathscr{A}$ is \emph{asymptotically optimal}, if it attains the Lai-Robbins lower bound \cite{lai1985asymptotically} in the sense that for any non-trivial problem instance $\Delta$,
	\begin{align}\label{ineq:lai_robbins_intro}
	\lim_{T\to \infty} \frac{\texttt{Reg}(\mathscr{A}) }{\log T} = \sum_{a: \Delta_a>0} \frac{2\sigma^2}{\Delta_a}.
	\end{align}
	\item[(R2)] 
    $\mathscr{A}$ is \emph{minimax optimal}, if it achieves the minimax regret \cite{auer2002nonstochastic} up to a universal constant $C>0$ in the sense that
	\begin{align}\label{ineq:minimax}
	\sup_{\Delta \in \R_{\geq 0}^K} \texttt{Reg}(\mathscr{A})\leq C\cdot \sigma\sqrt{KT}.
	\end{align}
\end{enumerate}
The viewpoints under these two paradigms are rather different: while (R1) asserts a precise logarithmic regret in the asymptotic regime $T\to \infty$ for any fixed bandit problem instance $\Delta$, the minimax approach (R2) pinpoints a substantially larger regret for a worst problem instance that can occur for any finite time horizon $T$. Due to this complementary nature of (R1)-(R2), a significant body of research has been devoted to understanding the (near) optimality properties of the regret of UCB algorithms; the readers are referred to \cite{lai1987adaptive,agrawal1995sample,katehakis1995sequential,auer1995gambling,auer2002finite,audibert2007tuning,audibert2009minimax,honda2010asymptotically,kaufmann2012bayesian,cappe2013kullback,degenne2016anytime,menard2017minimax,kaufmann2018bayesian,lattimore2018refining,hao2019bootstrapping,garivier2022kl,ren2024lai} for notable progress in this direction.

While the (near) validation of (R1)-(R2) for UCB algorithms provides important insights into their theoretical optimality, there remains a significant gap between the theory in (\ref{ineq:lai_robbins_intro})-(\ref{ineq:minimax}) and the actual algorithmic performance of the regret. For instance, in the most interesting case where the sub-optimality gaps $\{\Delta_a: \Delta_a>0\}$ are relatively small, although the very definition (\ref{def:regret_generic}) necessarily entails a relatively small regret, the Lai-Robbins formula (\ref{ineq:lai_robbins_intro}) predicts a large regret, whereas the maximum risk (\ref{ineq:minimax}) offers little instance-specific information. This discrepancy reveals an intrinsic limitation of both paradigms (R1) and (R2): these criteria can be overly conservative either for finite time horizons $T$ or for individual problem instance $\Delta$'s, and neither approach fully compensates for the weakness of the other.

The first goal of this paper is to offer a new theoretical paradigm for understanding the \emph{precise} regret behavior of a popular class of \texttt{UCB1} algorithms\footnote{See Section \ref{subsection:UCB_alg} for a precise description of this algorithm.} as named in \cite{auer2002finite}. This new framework addresses both the major limitations of the existing theoretical paradigms (R1)-(R2), and provides a positive answer to (Q1). In particular, we show in Theorem \ref{thm:regret_gaussian} that 
\begin{align}\label{ineq:regret_char_intro}
\frac{ \texttt{Reg}(\texttt{UCB1}) }{ \sum_{a: \Delta_a>0} \Delta_a n_{a;T}^\ast} \approx 1,
\end{align}
where $\{n_{a;T}^\ast\}_{a\in [K]}$ are deterministic quantities computable via a fixed point equation; see Section \ref{subsection:fpe} for a precise definition. While we have presented (\ref{ineq:regret_char_intro}) in an asymptotic fashion here for simplicity, our technical result in Theorem \ref{thm:regret_gaussian} is fully non-asymptotic with explicit error bounds between the two terms in  (\ref{ineq:regret_char_intro}).

As will be clear in Section \ref{section:main_results} below, our precise regret characterization (\ref{ineq:regret_char_intro}) not only accurately captures the actual behavior of the \texttt{UCB1} algorithm for finite time horizon $T$ and individual problem instance $\Delta$'s, but also offers significant new insights into the regimes in which the existing paradigms (R1)-(R2) are indeed informative for the \texttt{UCB1} algorithm. In particular, we will show that the Lai-Robbins regret formula (\ref{ineq:lai_robbins_intro}) provides precise approximation to the actual regret of the \texttt{UCB1} algorithm, if and only if the sub-optimality gap is strong enough with $\min_{a: \Delta_a>0} (\Delta_a/\sigma) \gg \sqrt{K\log T/T}$. Moreover, we will show that the maximal regret of the \texttt{UCB1} algorithm strictly deviates from the minimax regret by a logarithmic factor, and is therefore minimax sub-optimal in a strict sense for the general $K$-bandit problem.

As may be hinted in the formulae in (\ref{def:regret_generic}) and (\ref{ineq:regret_char_intro}), the technical crux of our regret characterization in (\ref{ineq:regret_char_intro}) lies in a new method to characterize the number of arm pulls $\{n_{a;T}\}_{a \in [K]}$ via the deterministic quantities $\{n_{a;T}^\ast\}_{a\in [K]}$. Indeed, we show in Propositions \ref{prop:E_n_aT} and \ref{prop:sub_n_aT} that such a characterization is possible both in probability and in expectation: uniformly in $a \in [K]$,
\begin{align}\label{ineq:ratio_intro}
\frac{n_{a;T}}{n_{a;T}^\ast}\approx 1\hbox{ with high probability, and } \, \frac{\E n_{a;T}}{n_{a;T}^\ast} \approx 1.
\end{align}
Interestingly, the control of the ratio in expectation exhibits qualitatively different features from its simpler in-probability counterpart. In particular, while an in-probability control of the ratio (\ref{ineq:ratio_intro}) is possible for a large regime of exploration rates in the \texttt{UCB1} algorithm, the expectation control must witness a fundamental barrier for the exploration rate to stay above $\sigma\sqrt{2 \log T}$ in order to align with Lai-Robbins asymptotic lower bound (\ref{ineq:lai_robbins_intro}). Such a discrepancy occurs fundamentally due to the heavy-tailed nature of $\{n_{a;T}\}_{a\in [K]}$ in spite of the presence of lighted-tail Gaussian noises. 

While we work out in this paper the characterization (\ref{ineq:ratio_intro}) only for the \texttt{UCB1} algorithm, our technical approach to (\ref{ineq:ratio_intro}) is applicable to a very general class of UCB index algorithms. We give in Section \ref{section:general_principle} an outline for the main principles of our new approach.

\subsection{Adaptive inference}

The problem of statistical inference for adaptively collected data has posed a significant challenge for decades. By now, it is classical knowledge (cf. \cite{dickey1979distribution,white1958limiting,white1959limiting,lai1982least}) that conventional asymptotic theory may fail when applied to data collected in a sequential manner. 
More recently, \cite{deshpande2018accurate,zhang2020inference,deshpande2023online,khamaru2021near,lin2023semi,lin2024statistical,ying2024adaptive} have shown---through a combination of theoretical results and
simulations---that similar issues may arise in the bandit setting that is our main interest in this paper. 

To address these inference challenges, two main approaches have been proposed in the literature:
\begin{enumerate}
    \item[(I1)] The first approach leverages the martingale nature of the data through the martingale central limit theorem (cf., \cite{hall1980martingale,dvoretzky1972asymptotic}).
    This method, possibly coupled with debiasing techniques to remove biases induced by sequential data collection, typically yields confidence intervals that are provably valid in the asymptotic limit of infinite data. The readers are referred to \cite{deshpande2018accurate,deshpande2023online,khamaru2021near,lin2023semi,lin2024statistical,ying2024adaptive,zhang2020inference,hadad2021confidence,bibaut2021post,zhan2021off,syrgkanis2023post} for applications of this technique in various concrete contexts.
    \item[(I2)] The second approach builds upon non-asymptotic concentration inequalities for self-normalized martingales (cf., ~\cite{de2004self,de2009self}). While this method offers confidence intervals that are valid for any sample size, these intervals are often considerably wider due to the use of theoretical concentration inequalities. The readers are referred to \cite{abbasi2011improved,shin2019bias,waudby2023anytime} and references therein for examples of this approach.
\end{enumerate}

The second goal of this paper is to construct valid confidence intervals for the unknown mean reward vector $\{\mu_a\}_{a \in [K]}$ in (\ref{eqn:bandit-model}) (and beyond) via UCB algorithms that blend the strengths of both (I1) and (I2), thereby providing a solution to question (Q2). 

While this goal may appear tangential to our theory in (\ref{ineq:regret_char_intro})-(\ref{ineq:ratio_intro}),  
we obtain in Section \ref{section:application}, through an essential use of the deterministic characterization for $\{n_{a;T}\}_{a\in [K]}$ in (\ref{ineq:ratio_intro}), quantitative central limit theorems (CLTs) for:
\begin{enumerate}
	\item [(i)] the empirical sample mean of the unknown rewards in the multi-armed bandit setting, and
	\item [(ii)] a general class of Ridge estimators (including the ordinary least squares estimator), where the arm means exhibit some underlying latent structures. 
\end{enumerate}
An immediate consequence of these quantitative CLTs is that statistical inference via UCB-type algorithms for the unknown parameters of interest in these models can be conducted as if the data were collected in the conventional, i.i.d. setting, while also enjoying the desired non-asymptotic theoretical guarantees that are exact in the asymptotic limit.

The connection between the deterministic characterization for $\{n_{a;T}\}_{a\in [K]}$ in (\ref{ineq:ratio_intro}) and the feasibility of adaptive inference for $\{\mu_a\}_{a\in [K]}$ in the bandit setting (\ref{eqn:bandit-model}) is conceptually related to \cite{lai1982least}, which pioneered the notion of  `\emph{stability}' of sample covariance as a key factor for valid statistical inference, albeit in a different linear regression context. As such, we believe our technical approach in (\ref{ineq:ratio_intro}) also opens a new door for other statistical inference applications in sequential decision-making problems with UCB algorithms and beyond.

\subsection{Organization}
 The rest of the paper is organized as follows. Section \ref{section:general_principle} outlines a general principle for our theory (\ref{ineq:regret_char_intro})-(\ref{ineq:ratio_intro}). In Section \ref{section:main_results}, we present our results on the precise characterizations of regret and pseudo-regret for the \texttt{UCB1} algorithm and its numerous consequences. In Section \ref{section:application}, we detail the aforementioned quantitative CLTs. The numerical performance of their associated CI's are reported in Section \ref{section:simulation}. Proofs of all our results are deferred to Sections \ref{section:proof_preliminary}-\ref{section:proof_applications} and the Appendix.

\subsection{Notation}\label{section:notation}

For any two integers $m,n \in \mathbb{Z}$, let $[m:n]\equiv \{m,m+1,\ldots,n\}$ when $m\leq n$ and $\emptyset$ otherwise. Let $[m:n)\equiv [m:n-1]$, $(m:n]\equiv [m+1:n]$, and we write $[n]\equiv [1:n]$. For $a,b \in \R$, $a\vee b\equiv \max\{a,b\}$ and $a\wedge b\equiv\min\{a,b\}$. For $a \in \R$, let $a_\pm \equiv (\pm a)\vee 0$. For $a>0$, let $\log_+(a)\equiv 1\vee \log(a)$. For $x \in \R^n$, let $\pnorm{x}{p}$ denote its $p$-norm $(0\leq p\leq \infty)$, and $B_{n;p}(R)\equiv \{x \in \R^n: \pnorm{x}{p}\leq R\}$. We simply write $\pnorm{x}{}\equiv\pnorm{x}{2}$ and $B_n(R)\equiv B_{n;2}(R)$. For a matrix $M \in \R^{m\times n}$, let $\pnorm{M}{\op}$ denote the spectral norm of $M$.

We use $C_{x}$ to denote a generic constant that depends only on $x$, whose numeric value may change from line to line unless otherwise specified. $a\lesssim_{x} b$ and $a\gtrsim_x b$ mean $a\leq C_x b$ and $a\geq C_x b$, abbreviated as $a=\bigo_x(b), a=\Omega_x(b)$ respectively;  $a\asymp_x b$ means $a\lesssim_{x} b$ and $a\gtrsim_x b$, abbreviated as $a=\Theta_x(b)$. For two nonnegative sequences $\{a_n\}$ and $\{b_n\}$, we write $a_n\ll b_n$ (respectively~$a_n\gg b_n$) if $\lim_{n\rightarrow\infty} (a_n/b_n) = 0$ (respectively~$\lim_{n\rightarrow\infty} (a_n/b_n) = \infty$). We follow the convention that $0/0 = 0$. $\bigo$ and $\smallo$ (resp. $\mathcal{O}_{\mathbf{P}}$ and $\mathfrak{o}_{\mathbf{P}}$) denote the usual big and small O notation (resp. in probability). 

For two real-valued random variables $X,Y$, their Kolmogorov distance $\mathfrak{d}_{ \mathrm{Kol} }(X,Y)$ is defined by
\begin{align}\label{def:d_kol}
\mathfrak{d}_{ \mathrm{Kol} }(X,Y)\equiv \sup_{t \in \R}\bigabs{\Prob(X\leq t)-\Prob(Y\leq t)}.
\end{align}

\section{Principles of our new approach to (\ref{ineq:regret_char_intro})-(\ref{ineq:ratio_intro}) }\label{section:general_principle}
In this section, we provide a heuristic description of our new analytic approach in a general setting where the multi-armed bandit problem is solved by an UCB index algorithm. 

\subsection{The general setting}
Consider UCB index algorithms of the following form. 
Let $\mu_a$ be the mean reward for arm $a$. 
For the $t$-th arm pull, the allocation is determined by 
\begin{align}\label{def:UCB-index-policy}
A_t \equiv \argmax_{a\in [K]} \texttt{ucb}(\hat{\mu}_{a;t-1},n_{a;t-1},t,T), 
\end{align}
where $\texttt{ucb}(\mu,n,t,T)$ is an index function, 
$\hat{\mu}_{a;t}$ is an estimate of $\mu_a$ 
and $n_{a;t}$ is the sample size from arm $a$ with the first $t$ pulls, and $T$ is a predetermined 
horizon for performance evaluation. In general, $\texttt{ucb}(\mu,n,t,T)$ is increasing in $(\mu,t,T)$ 
and decreasing in $n$. For `anytime' algorithms with unspecified $T$, 
$\texttt{ucb}(\mu,n,t,T)$ does not depend on $T$. 
Otherwise, we are typically able to write $\texttt{ucb}(\mu,n,T)$ as the UCB index 
function when it does not depend on $t$. 

The UCB index function in \eqref{def:UCB-index-policy} can be constructed by inverting 
the (minimum) Kullback-Leibler (KL) divergence between arms with means $\mu$ and $\mu_+$, 
\begin{equation}\label{def:gen-KL-UCB-index}
\texttt{ucb}(\mu,n,t,T) \equiv \argmax\big\{\mu_+: \mu_+\geq \mu,\ n\cdot \texttt{KL}(\mu,\mu_+) \leq g(T,t,n)\big\}
\end{equation}
for some non-negative exploration function $g(T,t,n)$ decreasing in $n$ and increasing in $(t,T)$.  
Such KL-based UCB index was first introduced in \cite{lai1987adaptive} 
for prespecified horizon $T$ and exploration functions of the form $g(T,t,n)=g(T/n)$ 
with $g(\cdot)$ a decreasing function. 
Thus, \eqref{def:gen-KL-UCB-index} is a slight generalization of Lai's UCB \cite{lai1987adaptive}. For $g(T,t,n) = \log T$, the UCB index in \eqref{def:gen-KL-UCB-index} is an inversion of the 
Lai-Robbins \cite{lai1985asymptotically} lower bound for the expected number of pulls because  
\begin{align*}
n\leq (\log T)/\texttt{KL}(\mu_a,\mu_\ast)\quad\Leftrightarrow\quad \texttt{ucb}(\mu_a,n,T) \geq \mu_\ast.
\end{align*}

\subsection{High-level ideas of our new analytic approach}
Our analytical approach is to compare the algorithm in \eqref{def:UCB-index-policy} 
with its noiseless continuous-time counterpart. 
In this noiseless continuous-time version, 
the sample size growth curves $n_{a;t}^\ast$ 
are increasing functions of $t$ satisfying  
\begin{align}\label{eq:noiseless-1}
\sum_{a\in [K]} n_{a;t}^* = t,\ \ 
\texttt{ucb}(\mu_a,n_{a;t}^\ast,t,T) = \mu^\ast_{+,t},\quad \forall a\in [K] 
\end{align}
for some function $\mu^\ast_{+,t}$. Importantly,
because $\texttt{ucb}(\mu,n,t,T)$ is decreasing in $n$, 
the noiseless version of \eqref{def:UCB-index-policy} can be viewed as a greedy 
minimax algorithm for $\texttt{ucb}(\mu_a,n_{a;t}^\ast,t,T)$, 
and its continuous-time version must have a common value $\mu^\ast_{+,t}$ of the UCB. 
Thus, by inversion of \eqref{eq:noiseless-1}, we may solve  $\mu^\ast_{+,t}$ as the unique solution of the equation
\begin{align}\label{eq:noiseless-2}
\sum_{a\in [K]} \sup\big\{n\in [0,t]: \texttt{ucb}(\mu_a,n,t,T) \geq \mu^\ast_{+,t}\big\} = t. 
\end{align}
For performance evaluation at $t=T$, the noiseless continuous-time version would provide 
a target regret 
\begin{align}\label{eq:noiseless-3}
\texttt{Reg}^\ast \equiv \sum_{a\in [K]} \Delta_a n^\ast_{a;T}. 
\end{align}
The objective of our analysis is to find proper conditions for the convergence of the 
UCB algorithm \eqref{def:UCB-index-policy} 
and its noiseless continuous-time counterpart in the sense of 
\begin{align}\label{eq:noiseless-4}
\frac{\texttt{Reg}(\mathscr{A})}{\texttt{Reg}^\ast}\approx 1,\quad 
 \max_{a\in [K]}\E\biggabs{\frac{n_{a;T}}{n^*_{a;T}}-1} \approx 0.
\end{align}
This would allow us to gain more precise insight 
about the performance of the UCB algorithm 
compared with typical existing results focused on performance upper bounds on the regret 
and the expected number of pulls from suboptimal arms. 

Clearly, the approach outlined above is both distribution dependent and algorithm dependent. In this paper, we confine ourselves to a definitive solution in the canonical Gaussian model with constant exploration function $g(T,t,n)$ in \eqref{def:gen-KL-UCB-index}. We expect our new analytic approach developed here to be broadly applicable to the general class of UCB index algorithms with further technical works.

\section{Precise regret}\label{section:main_results}

\subsection{The \texttt{UCB1} algorithm}\label{subsection:UCB_alg}
We consider a generalized version of the UCB algorithm \cite{lai1987adaptive}
that is parameterized by an exploration rate $\gamma_T>0$, hereafter referred to \texttt{UCB1}. To formally describe the \texttt{UCB1} algorithm, let
\begin{align}\label{def:mu_aT}
\hat{\mu}_{a;t}\equiv \frac{1}{n_{a;t}}\sum_{s \in [t]} R_s \bm{1}_{A_s=a},\quad (a,t) \in [K]\times [T]
\end{align}
be the empirical sample mean of $\mu_a$ up to round $t$. The \texttt{UCB1} algorithm is initialized by first pulling each arm once. Then, given the past information $\{(\hat{\mu}_{a;s},n_{a;s}): a \in [K], s \in [t-1]\}$ up to round $t-1$, the \texttt{UCB1} algorithm at round $t$ chooses an arm $A_t$ that maximizes the over-estimated mean $\hat{\mu}_{a;t-1}+\sigma \gamma_T/\sqrt{ n_{a;t-1} }$.
For the convenience of the reader, we summarize this generalized \texttt{UCB1} algorithm in Algorithm \ref{alg:ucb}. 

In the Gaussian model, 
the over-estimated mean in Algorithm \ref{alg:ucb} corresponds to Lai's UCB index in \eqref{def:gen-KL-UCB-index} 
with 
\begin{align}\label{eqn:ucb1_Lai}
\texttt{KL}(\mu,\mu_+)=\frac{(\mu_+-\mu)^2}{2\sigma^2},\, g(T,t,n)= \frac{\gamma_T^2}{2},\, \texttt{ucb}(\mu,n,t,T) = \mu+\frac{\sigma\gamma_T}{\sqrt{n}}.
\end{align}
We note that this algorithm operates far beyond the Gaussian model. In fact,
\cite{auer2002finite} derived a nonasymptotic upper bound of the regret for the natural choice of 
$\gamma_T = \sqrt{2\log T}$ and named the algorithm \texttt{UCB1}. 
We adopt here the same name \texttt{UCB1} for general $\gamma_T>0$.

\begin{algorithm}[t]
	\caption{\texttt{UCB1} algorithm with exploration rate $\gamma_T>0$}\label{alg:ucb}
	\begin{algorithmic}[1]
		\STATE Pull once each of the $K$ arms in the first $K$ iterations.
		\FOR{$ t = K+1$ to $T$}
		\STATE Choose the arm $A_t$ by 
		\begin{align*}
		A_{t} \in \argmax_{a \in [K]}\bigg\{ \hat{\mu}_{a;t-1}+ \frac{\sigma\cdot \gamma_T}{  \sqrt{n_{a;t-1}}} \bigg\}.
		\end{align*}
		\ENDFOR
	\end{algorithmic}
\end{algorithm}

\subsection{The fixed point equation and related quantities}\label{subsection:fpe}

The generic expected regret in (\ref{def:regret_generic}) for the \texttt{UCB1} algorithm will be written as 
\begin{align}\label{def:regret}
R(\Theta)\equiv \texttt{Reg}(\texttt{UCB1}),\quad \Theta\equiv (T,K,\Delta,\sigma,\gamma_T).
\end{align}
Here and below, we reserve the notation $\Theta$ to summarize all parameters in the bandit problem and the \texttt{UCB1} algorithm.

\begin{definition}\label{def:n_ast_0}
\begin{enumerate}
	\item Let $n_\ast(\Theta)$ be determined via the equation
	\begin{align}\label{def:n_ast}
	\sum_{a \in [K]}  n_\ast \big(1+{n_\ast^{1/2}\Delta_a}/(\sigma\gamma_T)\big)^{-2} = T.
	\end{align}
	\item With $n_\ast(\Theta)$ defined in (\ref{def:n_ast}), let
	\begin{align}\label{def:n_aT}
	n_{a;T}^\ast\equiv n_{\ast}(\Theta) \big(1+{n_{\ast}^{1/2}(\Theta)\Delta_a}/(\sigma\gamma_T)\big)^{-2},
	\quad a \in [K].
	\end{align}
	\item The theoretical regret $\texttt{Reg}_T^\ast(\Theta)$ for $\texttt{UCB1}$ is defined as 
	\begin{align}
	\texttt{Reg}_T^\ast(\Theta)\equiv \sum_{a \in [K]} \Delta_a n_{a;T}^\ast.
	\end{align}
\end{enumerate}

\end{definition}

\begin{remark}\label{rmk:n_ast}
Some interpretation and technical remarks:
\begin{enumerate}
	\item $n_\ast(\Theta)$ is related to the noiseless UCB in \eqref{eq:noiseless-1} 
	through $\mu^\ast_{+,T} = \mu_\ast+\sigma\gamma_T/\sqrt{n_\ast(\Theta)}$ with \eqref{eqn:ucb1_Lai}.
	By Lemma \ref{lem:n^*_at} below, for any $\sigma>0$, there exists a unique $n_\ast(\Theta) \in [T/K,T]$ solving the equation (\ref{def:n_ast}). Moreover, as the map $n\mapsto \sum_{a \in [K]} \big(1/{\sqrt{n}}+{\Delta_a}/(\sigma \gamma_T)\big)^{-2}$ is strictly increasing, $n_\ast(\Theta)$ can be easily solved numerically by bisection methods.
	\item  $n_{a;T}^\ast=(\gamma_T^2/2)/\texttt{KL}(\mu_a,\mu^\ast_{+,T})$ 
	is an explicit version of the noiseless sample size in \eqref{eq:noiseless-1} 
	through \eqref{eqn:ucb1_Lai}. As will be rigorously proven in Section \ref{section:main_results} ahead, the quantity $n_{a;T}^\ast$ is ratio consistent for $n_{a;T}$ both in probability and in expectation.
\end{enumerate}
\end{remark}

We define two more quantities that will be used repeatedly below:
\begin{itemize}
	\item First, the error $\err(\Theta)$ is defined by
	\begin{align}\label{def:err_regret_gaussian}
	\err(\Theta) \equiv    \frac{\sqrt{\log \gamma_T}+\sqrt{\log \log T}}{\gamma_T}+\frac{K}{T}+\frac{\pnorm{\Delta/\sigma}{\infty}^2}{\gamma_T^2}.
	\end{align}
	\item Next, we define 
	\begin{align}\label{def:D_ast}
	D_\ast \equiv \frac{1}{T}\sum_{a\in [K]}\bigg(\frac{n^\ast_{a;T}}{n_\ast}\bigg)^{1/2}n^\ast_{a;T}.
	\end{align}
	As will be clear from Lemma \ref{lem:n^*_at} below, $D_\ast \in [K^{-1/3},1]$. 
\end{itemize}

\subsection{Characterization of the expected regret $R(\Theta)$}
\label{sec:regret}

We will work with Gaussian noises $\{\xi_t\}$ in this paper, formally recorded as below.

\begin{assumption}\label{assump:gaussian_noise}
	The scaled noises $\{\xi_t\}_{t \in [T]}$ are i.i.d. $\mathcal{N}(0,1)$.
\end{assumption}

Our first main result below characterizes the behavior of  $R(\Theta)$ in (\ref{def:regret}).

\begin{theorem}\label{thm:regret_gaussian}
    Suppose Assumption \ref{assump:gaussian_noise} holds. There exists a universal constant $C>0$ such that if $\err(\Theta)\leq 1/C$, then 
	\begin{align*}
	\biggabs{ \frac{R(\Theta)}{  \texttt{Reg}_T^\ast(\Theta) }-1  }\leq C \cdot \big(\err(\Theta)+\vartheta_T^\ast\big).
	\end{align*}
	Here $\vartheta_T^\ast\equiv \gamma_T^{-2}\cdot T e^{-\gamma_T^2/2}$.
\end{theorem}

The proof of Theorem \ref{thm:regret_gaussian} leads to the following \emph{two-sided} expectation control for $\E n_{a;T}$. In fact, we provide a slight stronger control for the first moment of the difference $\abs{n_{a;T}/n_{a;T}^\ast-1}$ that will be useful in the applications in Section \ref{section:application}.

\begin{proposition}
\label{prop:E_n_aT}
	Suppose Assumption \ref{assump:gaussian_noise} holds. There exists a universal constant $C>0$ such that if $\err(\Theta) \leq 1/C$, then 
	\begin{align*}
	\max_{a \in [K]} \E \,\abs{n_{a;T}/n_{a;T}^\ast-1}\leq C\cdot \big(D_\ast^{-1}\err(\Theta)+\vartheta_T^\ast\big).
	\end{align*}
	Here recall $D_\ast$ is defined in (\ref{def:D_ast}), and $\vartheta_T^\ast= \gamma_T^{-2}\cdot T e^{-\gamma_T^2/2}$.
\end{proposition}

The proofs of both Theorem \ref{thm:regret_gaussian} and Proposition \ref{prop:E_n_aT} can be found in Section \ref{section:proof_regret_pseudo}.

\begin{remark}\label{rmk:regret_expected}
Some technical remarks on the conditions and the error terms in Theorem \ref{thm:regret_gaussian} and Proposition \ref{prop:E_n_aT}:
\begin{enumerate}
    \item The constant $C>0$ can be tracked easily in the proof; we refrain from doing so to keep the presentation and the proof clean. 
    \item For the canonical choice $\gamma_T=\sqrt{2\log T}$, $\err(\Theta)$ vanishes if and only if $K/T\to 0$ and $\pnorm{\Delta/\sigma}{\infty}/\gamma_T\to 0$. The formal condition is minimal, and the latter condition essentially ensures that the (expected) number of arm pulls diverges as $T \to \infty$. Similar conditions have been adopted in, e.g., \cite{lai1985asymptotically,lai1987adaptive}.
       \item 
    $\gamma_T$ cannot drop substantially below $\sqrt{2\log T}$ for the error terms to be small due to the the crucial factor $\vartheta_T^\ast$. In fact, if such small $\gamma_T$ is allowed, then the asymptotic Lai-Robbins lower bound (\ref{ineq:lai_robbins}) will be violated.
    \item The error bound for $\E n_{a;T}/n_{a;T}^\ast-1$ comes with an additional $D_\ast^{-1}$. As will be clear from Lemma \ref{lem:n^*_at}, this factor arises for optimal arms, but it does not appear in the regret that takes average over sub-optimal arms. 
    \item The Gaussianity assumption on the noises is used in the boundary crossing probability in Lemma \ref{lem:BM_crossing}. 
    We believe our results hold in a weaker form under the first and second moment assumptions when the left tail of the distribution of the awards from one of the optimal arm is sub-Gaussian.
\end{enumerate}
\end{remark}

\subsubsection{Relation to the Lai-Robbins bound}
Recall that in our setting, the Lai-Robbins lower bound \cite{lai1985asymptotically} states that for any fixed $K \in \N$ and $\{\Delta_a\}\subset \R^K_{\geq 0}$, and any choice of the exploration rate $\{\gamma_T\}$, the following hold: 
\begin{align}\label{ineq:lai_robbins}
\liminf_{T\to \infty} \frac{R(\Theta)}{ 2\log T} \geq  \sum_{a: \Delta_a>0} \frac{\sigma^2}{\Delta_a}.
\end{align}
When $\{\Delta_a\}$ becomes small, $R(\Theta)$ is also small non-asymptotically, so the asymptotic lower bound (\ref{ineq:lai_robbins}) becomes uninformative. Our Theorem \ref{thm:regret_gaussian} provides a precise characterization of the magnitude of $\{\Delta_a\}$ for which the Lai-Robbins bound (\ref{ineq:lai_robbins}) remains informative for the actual expected regret $R(\Theta)$. 

To this end, we consider the following parameter space. Fix $\eta_T>0$. For any $L>0$ and $K \in \mathbb{N}$, let
\begin{align*}
\mathscr{C}(L)&\equiv \Big\{\Delta\in \R^K: \mathrm{supp}(\Delta)\subsetneq [K],  \{\Delta_a/\sigma\}_{a: \Delta_a>0} \subset \big[L \sqrt{K\log T/T},  \eta_T\sqrt{\log T}\big]\Big\}. 
\end{align*}

\begin{corollary}\label{cor:lr_bound}
	Suppose Assumption \ref{assump:gaussian_noise} holds. Consider the \texttt{UCB1} allocation rule with $\gamma_T\equiv \sqrt{2 \log T}$. Let $
	R^{\mathsf{LR}}(\Theta)\equiv 2\log T\cdot \sum_{a: \Delta_a>0} ({\sigma^2}/{\Delta_a})$. 
	 Then there exists some universal constant $C>0$ such that for $K/T\leq 1/C$, the following hold:
	\begin{enumerate}
		\item For any $L>C$, with $\epsilon_{T,K}\equiv \sqrt{\log \log T/\log T}+K/T$,
		\begin{align*}
		\sup\nolimits_{ \Delta \in \mathscr{C}(L) } \abs{ {R(\Theta)}/{ R^{\mathsf{LR}}(\Theta) }-1}\leq C\cdot \big( 1/L+\epsilon_{T,K}+\eta_T\big).
		\end{align*}
		\item For any $\epsilon \in (0,1/2)$, 
		\begin{align*}
		\inf\nolimits_{ \Delta \in \mathscr{C}(\epsilon) } {R(\Theta)}/{ R^{\mathsf{LR}}(\Theta) }\leq C\cdot\epsilon^2.
		\end{align*}
	\end{enumerate}
\end{corollary}

The above corollary shows that $\sqrt{K\log T/T}$ is the crucial scaling of the `signal strength' $\min_{a: \Delta_a>0}(\Delta_a/\sigma)$ for the Lai-Robbins bound (\ref{ineq:lai_robbins}) to be informative:
\begin{enumerate}
	\item If $\min_{a: \Delta_a>0}(\Delta_a/\sigma)\gg \sqrt{K\log T/T}$, then the Lai-Robbins bound (\ref{ineq:lai_robbins}) provides a uniformly accurate characterization of the regret $R(\Theta)$ of the \texttt{UCB1} algorithm.
	\item If $\min_{a: \Delta_a>0}(\Delta_a/\sigma)\lesssim \sqrt{K\log T/T}$, then the Lai-Robbins bound (\ref{ineq:lai_robbins}) can be arbitrarily loose for the regret $R(\Theta)$ of the \texttt{UCB1} algorithm. 
\end{enumerate}

\subsubsection{Relation to the maximal regret}
Let us now apply Theorem \ref{thm:regret_gaussian} to obtain a lower bound for the maximum regret $\sup_{\Delta } R(\Theta)$.

\begin{corollary}\label{cor:minimax_risk}
	Suppose Assumption \ref{assump:gaussian_noise} holds. Consider the \texttt{UCB1} allocation rule with with $\gamma_T=\sqrt{2\log T}$. Then there exists some universal constant $C>0$ such that for $K/T\leq 1/C$, with $\epsilon_{T,K}\equiv \sqrt{\log \log T/\log T}+K/T$, 
	\begin{align*}
	\sup\nolimits_{\Delta \in \R_{\geq 0}^K } (R(\Theta)/\sigma)\geq \big(1/\sqrt{8}-C\cdot \epsilon_{T,K}\big)_+ \sqrt{TK \log T}.
	\end{align*}
\end{corollary}

The above corollary shows that the maximum risk of the \texttt{UCB1} algorithm deviates from the minimax regret (\ref{ineq:minimax}) by a multiplicative factor of  $\sqrt{\log T}$. The only related result appears to be \cite[Theorem 4]{kalvit2021closer}, which rigorously confirms the minimax sub-optimality of \texttt{UCB1} in the two-armed setting. To the best of our knowledge, our lower bound for general $K$ in the above corollary is new. 

The proofs of both Corollaries \ref{cor:lr_bound} and \ref{cor:minimax_risk} can be found in Section \ref{section:proof_regret_pseudo}.

\subsection{Characterization of the pseudo-regret $\hat{R}^{\mathsf{p}}(\Theta)$}
\label{sec:pseudo-regret}

We now complement the results in the previous section by studying their stochastic regret versions. The vanilla  \emph{regret} over $T$ rounds is given by
\begin{align}\label{def:regret_random}
\hat{R}(\Theta)\equiv \sum_{t \in [T]} \big(\mu_\ast-R_t\big)= \sum_{t \in [T]} \Delta_{A_t} - \sigma \sum_{t \in [T]} \xi_t. 
\end{align}
As the fluctuation of the random error in (\ref{def:regret_random}) is typically of order $\bigop(T^{-1/2})$ whereas the expected regret $R(\Theta)$ in (\ref{def:regret}) are typically of a substantially lower order, we shall consider a version of (\ref{def:regret_random}) by removing its random error component to shed light on the underlying stochastic behavior of (\ref{def:regret}). Borrowing the terminology from \cite{audibert2009exploration}, let the \emph{pseudo-regret} over $T$ rounds be
\begin{align}\label{def:regret_pseudo}
\hat{R}^{\mathsf{p}}(\Theta)\equiv \sum_{t \in [T]} \Delta_{A_t} = \sum_{a: \Delta_a>0} \Delta_a  n_{a;T}.
\end{align}
Clearly, $R(\Theta)$ and $\hat{R}^{\mathsf{p}}(\Theta)$ can be related via $\E \hat{R}^{\mathsf{p}}(\Theta)=R(\Theta)$. 

Our second main result below provides a characterization for the  pseudo-regret $\hat{R}^{\mathsf{p}}(\Theta)$ in (\ref{def:regret_pseudo}). 

\begin{theorem}\label{thm:regret_pseudo}
    Suppose Assumption \ref{assump:gaussian_noise} holds. There exists a universal constant $C>0$ such that if $\err(\Theta)\leq 1/C$, then 
	\begin{align*}
	\Prob\bigg(\biggabs{ \frac{\hat{R}^{\mathsf{p}}(\Theta)}{  \texttt{Reg}_T^\ast(\Theta)  }-1  }\geq C\cdot \err(\Theta) \bigg)\leq C\cdot \gamma_T^{-100}.
	\end{align*}
\end{theorem}

Interestingly, the pseudo-regret $\hat{R}^{\mathsf{p}}(\Theta)$ can be significantly smaller in order than the expected regret $R(\Theta)=\E \hat{R}^{\mathsf{p}}(\Theta)$. For instance, for $\gamma_T=(\log \log T)^{0.51}$, the above theorem  implies that $\hat{R}^{\mathsf{p}}(\Theta) = \bigop\big(  (\log \log T)^{1.02} \big)$ for fixed $\Delta$, whereas the expected regret must satisfy $\E \hat{R}^{\mathsf{p}}(\Theta) = \Omega\big(\log T\big)$ due to the Lai-Robbins lower bound (\ref{ineq:lai_robbins}). The discrepancy between $\hat{R}^{\mathsf{p}}(\Theta)$ and $\E\hat{R}^{\mathsf{p}}(\Theta)$  has been previously observed in \cite{coquelin2007bandit,abbasi2011improved} for a class of UCB algorithms that take the confidence level as an input. Here our results show that such phenomenon continues to hold for the vanilla \texttt{UCB1} algorithm with small $\gamma_T$'s.

Similar to Proposition \ref{prop:E_n_aT}, the proof of Theorem \ref{thm:regret_pseudo} leads to a ratio control for $n_{a;T}/n_{a;T}^\ast$ that now holds in probability.

\begin{proposition}\label{prop:sub_n_aT}	
	Suppose Assumption \ref{assump:gaussian_noise} holds. Fix $\gamma\geq 1$. There exists a universal constant $C>0$ such that if $\err(\Theta)\vee (\gamma/\gamma_T)\leq 1/C$, then 
	\begin{align*}
	\max_{a \in [K]} \Prob\big[\abs{ {n_{a;T}}/{n_{a;T}^\ast}-1}\geq C\cdot \big(D_\ast^{-1} \err(\Theta)+\gamma/\gamma_T\big) \big]\leq C\cdot \big(\gamma_T^{-100}+\log T\cdot \gamma e^{-\gamma^2/2}\big).
	\end{align*}
	Here recall $D_\ast$ is defined in (\ref{def:D_ast}).
\end{proposition}

The proofs of both Theorem \ref{thm:regret_pseudo} and Proposition \ref{prop:sub_n_aT} can be found in Section \ref{section:proof_regret_pseudo}.

\begin{remark}
Proposition \ref{prop:sub_n_aT} above recovers several results in \cite{kalvit2021closer} under a slight variant of the \texttt{UCB1} algorithm. 

\begin{enumerate}
    \item In the two-armed case $K=2$, using the notation of \cite{kalvit2021closer}, with $\sigma=1$, $\gamma_T=\sqrt{\rho \log T}$ and $\Delta = \sqrt{\theta \log T/T}$, the fixed point equation (\ref{def:n_ast}) becomes
\begin{align*}
n_\ast + \bigg(\frac{1}{\sqrt{n_\ast} }+ \sqrt{ \frac{\theta}{\rho T} }\bigg)^{-2} = T\,\Leftrightarrow\, \frac{1}{\sqrt{1-n_\ast/T}}-\frac{1}{\sqrt{n_\ast/T}} = \sqrt{ \frac{\theta}{\rho} }.
\end{align*}
Compared with \cite[Eqn. 2]{kalvit2021closer}, by setting $\lambda_\rho^\ast(\theta)\equiv n_\ast/T$ in the above equation, Proposition \ref{prop:sub_n_aT} recovers \cite[Theorem 1 and Theorem 4]{kalvit2021closer} respectively with a stronger non-asymptotic estimate.
\item In the $K$-armed case, \cite{kalvit2021closer} considered two special cases:
\begin{enumerate}
    \item In the first case, $\Delta_a=0$ for all $a \in [K]$. Then it is easy to solve by (\ref{def:n_ast}) that $n_\ast=T/K$, so Proposition \ref{prop:sub_n_aT} recovers \cite[Theorem 2-(1)]{kalvit2021closer}.
    \item In the second case, $K$ is fixed and $\min_{a:\Delta_a>0} (\Delta_a/\sigma)\gg  \sqrt{\log T/T}$. In this case, the summation in (\ref{def:n_ast}) over sub-optimal arms trivializes with $n_\ast \approx T/\abs{ \{a: \Delta_a=0\} }$. Proposition \ref{prop:sub_n_aT} then recovers \cite[Theorem 2-(2)]{kalvit2021closer}.
\end{enumerate}
\end{enumerate}
\end{remark}

\section{Adaptive inference}\label{section:application}

\subsection{Inference for the mean rewards}

\subsubsection{Stability and quantitative central limit theorems}

The following notion of `\emph{stability}', directly motivated from \cite{lai1982least} in a slightly different context, provides a key connection of our theoretical results in Section \ref{section:main_results} and adaptive inference for the mean reward $\{\mu_a\}_{a\in [K]}$. 

\begin{definition}
We call a bandit algorithm $\mathscr{A}$ \emph{weakly-stable}, if for some deterministic real numbers $\{ n_{a;T}^\star (\mathscr{A})\}_{a \in [K]}$, the number of arm pulls $\{n_{a;T}(\mathscr{A})\}_{a \in [K]}$ by the algorithm~$\mathscr{A}$ satisfies
\begin{align}\label{eqn:weak-stability}
 \frac{n_{a;T}(\mathscr{A})}{n_{a;T}^\star(\mathscr{A})} \stackrel{\mathbb{P}}{\to} 1, \quad \text{for all arm} \;\; a \in [K].
\end{align}
We further call $\mathscr{A}$ \emph{strongly-stable}, if 
\begin{align}
    \label{eqn:strong-stability}
   \max_{a \in [K]} \E \biggabs{ \frac{n_{a;T}(\mathscr{A})}{n_{a;T}^\star(\mathscr{A})}  - 1 } \leq \err_{\mathscr{A}}(\Theta),
 \end{align}
 where $\err_{\mathscr{A}}(\Theta)$ is some algorithmic-specific error term that vanishes as $T \to \infty$.
\end{definition}

It is easy to show, following an argument of~\cite[Theorem~3]{lai1982least}, that for any weakly-stable bandit algorithm $\mathscr{A}$, the sample mean is asymptotically normal:
\begin{align}
\label{eqn:normality-demo}
    \big(n_{a;T}(\mathscr{A})/\sigma^2\big)^{1/2}\cdot (\widehat{\mu}_{a;T}(\mathscr{A}) - \mu_a) \stackrel{d}{\to} \mathcal{N}(0,1)\quad \hbox{ as }T\to \infty.
\end{align}
From here, our Proposition \ref{prop:sub_n_aT} immediately implies that under the conditions assumed therein, the central limit theorem (\ref{eqn:normality-demo}) holds for the \texttt{UCB1} algorithm. In fact, we may prove strong stability of \texttt{UCB1} in the non-asymptotic sense of (\ref{eqn:strong-stability}) by leveraging the expectation control in Proposition~\ref{prop:E_n_aT}. Recall $\mathfrak{d}_{ \mathrm{Kol} }$ defined in (\ref{def:d_kol}).

\begin{theorem}\label{thm:mu_aT_CLT}
Suppose Assumption \ref{assump:gaussian_noise} holds and $\gamma_T\leq (\log T)^{100}$. Then there exists a universal constant $C>0$ such that 
\begin{align*}
\max_{a \in [K]}\mathfrak{d}_{ \mathrm{Kol} }\Big((n_{a;T}/\sigma^2)^{1/2}\big(\hat{\mu}_{a;T}-\mu_a\big),\mathcal{N}(0,1)\Big)\leq C\cdot \big(D_\ast^{-1}\err(\Theta)+\vartheta_T^\ast\big)^{1/3}.
\end{align*}
Here $(\err(\Theta), D_\ast)$ are defined in (\ref{def:err_regret_gaussian})-(\ref{def:D_ast}) and $\vartheta_T^\ast=\gamma_T^{-2}\cdot Te^{-\gamma_T^2/2}$.
\end{theorem}

The proof of the above theorem can be found in Section \ref{section:proof_applications}. As an immediate application, we may construct confidence intervals for $\mu_a$ using the sample mean $\hat{\mu}_{a;T}$ with sample size $n_{a;T}$ as if the data were collected in an i.i.d. fashion: for any $\alpha \in (0,1)$, consider the following $(1-\alpha)$ confidence interval for $\mu_a$:
\begin{align}\label{def:CI_mean}
\mathsf{CI}_a(\alpha)\equiv \big[\hat{\mu}_{a;T}\pm z_{\alpha/2}\cdot \sigma n_{a;T}^{-1/2}\big].
\end{align}
Here and below, whenever confidence intervals are concerned, we reserve the notation $z_{\alpha}$ as the normal $1-\alpha$ quantile in that $\Prob(\mathcal{N}(0,1)\leq z_\alpha)=1-\alpha$. The following corollary follows directly from Theorem \ref{thm:mu_aT_CLT} above, so we omit a detailed proof.
\begin{corollary}\label{cor:mu_aT_CI}
Consider the setting in Theorem \ref{thm:mu_aT_CLT}. Then there exists some universal constant $C>0$ such that
\begin{align*}
\max_{a \in [K]}\sup_{\alpha \in (0,1)}\bigabs{\Prob\big(\mu_a \in \mathsf{CI}_a(\alpha)\big)-(1-\alpha)}\leq C\cdot \big(\err(\Theta)+\vartheta_T^\ast\big)^{1/3}.
\end{align*}
\end{corollary}

In the above discussion we have assumed known noise level $\sigma>0$ for simplicity. If the noise level is unknown, we may use any consistent estimator $\hat{\sigma}$ of $\sigma$ for practical implementation. A simple choice in our setting is
\begin{align*}
\hat{\sigma}^2\equiv \frac{1}{K}\sum_{a \in [K]} \bigg(\frac{1}{n_{a;T}}\sum_{t \in [T]} (R_t-\hat{\mu}_{a;T})^2\bm{1}_{A_t = a}\bigg).
\end{align*}
In fact, each term in the parentheses is a consistent estimator of the variance for arm $a \in [K]$, and therefore in the current homogeneous noise setting we may further take an average of these individual estimates to reduce the variability. These ramifications will not be further pursued for simplicity of presentation. 

\subsection{Inference for bandits with structured mean rewards}\label{subsection:inference_linear}
\subsubsection{The setup and the algorithm}
Consider the linear regression model, where we observe data $\{(x_t,y_t) \in \R^d\times \R: t \in [T]\}$, consisting of the covariates $\{x_t\}$ and associated responses $\{y_t\}$, which are related through
\begin{align}
y_t = \iprod{x_t}{\beta_\ast}+\xi_t,\quad t \in [T].
\end{align}
Here $\beta_\ast \in \R^d$ is the unknown parameter of interest, and $\{\xi_t\}$'s represent the additive noises. For notational convenience, we define the design matrix $X \in \R^{T \times d}$, where each row corresponds to $x_t^\top$, the response vector $Y = (y_1, \ldots, y_T) \in \R^T$, and the error vector $\xi = (\xi_1, \ldots, \xi_T) \in \R^T$.

We consider here a special case: suppose the covariates $\{x_t\}$ take values in a finite decision set $\mathscr{X}\equiv \{z_1,\ldots,z_K\}\subset \R^d$, and at round $t$, the covariate $x_t\equiv z_{A_t}$ is selected via the \texttt{UCB1} algorithm in Algorithm \ref{alg:ucb}, where $\hat{\mu}_{a;t}$ therein is now replaced by $\hat{y}_{a;t}\equiv n_{a;t}^{-1}\sum_{s \in [t]} y_s \bm{1}_{A_s=a}$.  Put simply, for a relatively small size of the decision set $\mathscr{X}$, the selection rule of this algorithm treats the problem as a $K$-arm bandit problem with arm means $\{ \iprod{z_a}{\beta_\ast} \}_{a \in [K]}$\footnote{When the size of the decision set $\mathscr{X}$ is large, one may instead use a contextual bandit algorithm, cf. \cite{abbasi2011improved,lattimore2020bandit}.}. It is easy to see that the expected regret of this algorithm
 \begin{align}
    R^{\mathrm{lin}}(\Theta) = \E \sum_{t \in [T]}  \Big(\max_{z \in \mathscr{X}}\iprod{z}{\beta_\ast} -
     \iprod{x_t}{\beta_\ast} \Big)
 \end{align}
is the same (up to logarithmic factors) as the contextual bandit algorithms~\cite{abbasi2011improved} when $|\mathscr{X}| = \bigo(d)$. The key difference in the current setting is that we are interested in the unknown regression coefficient vector $\beta_\ast \in \R^d$ rather than the arm means $\{ \iprod{z_a}{\beta_\ast} \}_{a \in [K]}$.

\subsubsection{The Ridge estimator}
Once the data $\left\{(x_t, y_t)\right\}_{t \in [T]}$ is collected, we can estimate the parameter $\beta_\ast$ of interest as via a Ridge regression estimator:
\begin{align}\label{def:ridge}
\hat{\beta}_\lambda =\argmin_{\beta \in \R^d} \bigg\{\frac{1}{T}\pnorm{Y-X\beta}{}^2+\lambda \pnorm{\beta}{}^2\bigg\}=\frac{1}{T}\big(S_T+\lambda I\big)^{-1}X^\top Y,
\end{align}
where $
S_T\equiv T^{-1} X^\top X=T^{-1}\sum_{t \in [T]} x_tx_t^\top \in \R^{d\times d}$
is the sample covariance. The case $\lambda=0$ in (\ref{def:ridge}) is in general understood as the limit of the right hand side as $\lambda \downarrow 0$. When $S_T$ is invertible, this limit $\hat{\beta}_0$ coincides with the ordinary least squares estimator. Conventional theory under i.i.d. observations $\{(x_t,y_t)\}_{t \in [T]}$ shows that, if the covariates $\{x_t\}$ are fixed, then (see e.g., Eqn. (\ref{ineq:ridge_CLT_0}))
\begin{align}\label{eqn:ridge_dist}
\sqrt{T/\sigma^2}\cdot  S_T^{-1/2} (S_T+\lambda I)\big(\hat{\beta}_\lambda-(S_T+\lambda I)^{-1} S_T \beta_\ast\big)\stackrel{d}{\approx} \mathcal{N}(0,I).
\end{align}
For random designs, $S_T$ can be replaced by the popular covariance matrix in low dimensions, but would otherwise require a mean-field modification in high dimensions, cf. \cite{dobriban2018high,hastie2022surprises,han2023distribution,bao2023leave}.

\subsubsection{Quantitative central limit theorems}
Let $
\Delta_a^{\mathrm{lin}}\equiv \max_{a' \in [K]} \iprod{z_a'}{\beta_\ast}-\iprod{z_a}{\beta_\ast}$ for all $ a \in [K]$. Let $\Theta^{\mathrm{lin}}$ be defined as (\ref{def:regret}), $n_\ast^{\mathrm{lin}}$ be the solution to (\ref{def:n_ast}), and $n_{a;T}^{\ast,\mathrm{lin}}$ be defined as (\ref{def:n_aT}), all upon replacing $\{\Delta_a\}$ therein with $\{\Delta_a^{\mathrm{lin}}\}$. Furthermore, let $
S_T^\ast\equiv T^{-1}\sum_{a \in [K]} n_{a;T}^{\ast,\mathrm{lin}} z_az_a^\top$
be the `population' version of $S_T$.

\begin{theorem}\label{thm:ridge_CLT}
Consider the above linear regression setting with $\mathrm{span}(\mathscr{X})=\R^d$. Suppose Assumption \ref{assump:gaussian_noise} holds and $\gamma_T\leq (\log T)^{100}$. Then there exists some universal constant $C>0$ such that uniformly in $w \in \partial B_d(1)$,
\begin{align*}
&\mathfrak{d}_{ \mathrm{Kol} }\Big( \bigiprod{w}{ \sqrt{T/\sigma^2}\cdot  S_T^{-1/2} (S_T+\lambda I)\big(\hat{\beta}_\lambda-(S_T+\lambda I)^{-1} S_T \beta_\ast\big)}, \mathcal{N}(0,1)\Big)\\
&\leq C \cdot \big(D_\ast^{\mathrm{lim},-1}\err(\Theta^{\mathrm{lin}})+\vartheta_T^\ast\big)^{1/3}\cdot \big[1+ \big(\mathsf{z}_K/\sigma_T^\ast\big)^{1/3}\big].
\end{align*}
Here $\mathsf{z}_K\equiv T^{-1} \sum_{a \in [K]} n_{a;T}^{\ast,\mathrm{lin}}\pnorm{z_a}{}^2$, $\sigma_T^\ast\equiv \lambda_{\min}(S_T^\ast)$, $\big(D_\ast^{\mathrm{lin}},\err(\Theta^{\mathrm{lin}})\big)$ are defined via (\ref{def:err_regret_gaussian})-(\ref{def:D_ast}) with $\Theta$ therein replaced by $\Theta^{\mathrm{lin}}$ and $\vartheta_T^\ast=\gamma_T^{-2}\cdot Te^{-\gamma_T^2/2}$.
\end{theorem}

\vspace{10pt}

The proof of the above theorem can be found in Section \ref{section:proof_applications}. 
 Note that in the above theorem the decision set $\mathscr{X}=\{z_a\}_{a \in [K]}$ is only required to expand $\R^d$ without orthogonality requirements (which would otherwise trivialize the problem). 

Similar to the development in the previous subsection, Theorem \ref{thm:ridge_CLT} above can be easily used to validate the `conventional' confidence intervals for the low-dimensional projections $\{\iprod{w}{\beta_\ast}\}_{w \in \partial B_d(1)}$. We omit these repetitive details.

\begin{remark}
In a related work, \cite{zhang2020inference} observes the failure of the central limit theorem for the ordinary least squares estimator in a Bernoulli batched bandit problem, in the asymptotic regime of the large batch limit with a \emph{fixed} horizon $T$. Here our Theorem \ref{thm:ridge_CLT} asserts a quantitative CLT in the large horizon limit $T\to \infty$. However, the distributional convergence to normal can be fairly slow, as empirically observed previously in a related two-arm bandit problem, cf. \cite[Figure 2]{deshpande2018accurate}.
\end{remark}

\section{Some illustrative simulations}\label{section:simulation}

For numerical illustration, we consider a two-armed bandit setting $K=2$ with sub-optimality gap $\Delta\geq 0$. The noise level is set as $\sigma=0.1$. All simulation results below are based on $B=1000$ Monte Carlo averages.

\subsection{Regrets and arm pulls}

\begin{figure}[t]
	\begin{minipage}[t]{0.495\textwidth}
		\includegraphics[width=\textwidth]{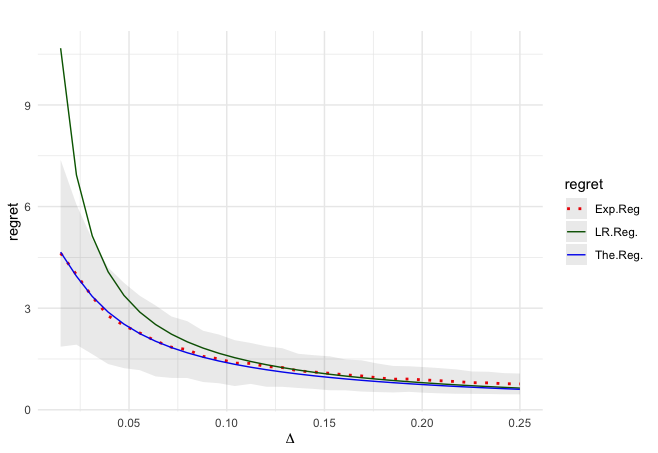}
	\end{minipage}
	\begin{minipage}[t]{0.495\textwidth}
		\includegraphics[width=\textwidth]{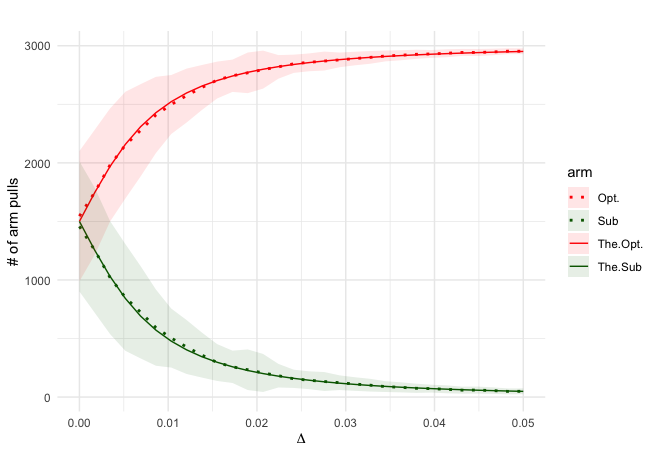}
	\end{minipage}
	\caption{\emph{Left panel}: Regrets for various sub-optimality gaps $\Delta \in [0.01,0.25]$. \emph{Right panel}: Number of arm pulls for various sub-optimality gaps $\Delta \in [0.01,0.05]$. }
	\label{fig:regret}
\end{figure}

In the first simulation, we will numerically validate:
\begin{enumerate}
	\item Precise regret characterizations in Theorems \ref{thm:regret_gaussian} and \ref{thm:regret_pseudo}.
	\item Proposition \ref{prop:E_n_aT}/\ref{prop:sub_n_aT} that provide characterizations for the number of arm pulls.
\end{enumerate}  
The exploration rate is set to be $\gamma_T=\sqrt{2\log T}$ as in the \texttt{UCB1} allocation rule with time horizon $T=3000$ for this simulation.

In the left penal of Figure \ref{fig:regret}, we numerically  validate (1) by plotting:
\begin{itemize}
	\item the expected regrets $R(\Theta)$ in the dotted red line via Monte-Carlo average of the pseudo-regrets $\hat{R}^{\mathsf{p}}(\Theta)$ in (\ref{def:regret_pseudo}); 
	\item the theoretical regrets $\sum_{a:\Delta_a>0} \Delta_a n_{a;T}^\ast $ in the blue line;
	\item  the Lai-Robbins regrets $R^{\mathsf{LR}}(\Theta)=\gamma_T^2\sum_{a:\Delta_a>0} (\sigma^2/\Delta_a)$ in the green line.
\end{itemize}
 The expected regret $R(\Theta)$ closely aligns with the theoretical regret formula obtained in Theorems \ref{thm:regret_gaussian} and \ref{thm:regret_pseudo} across different choices of the sub-optimality gap $\Delta$. Moreover, the Lai-Robbins asymptotic regret $R^{\mathsf{LR}}(\Theta)$ only aligns with the expected regret $R(\Theta)$ for large values of $\Delta$; this matches the theory in Corollary \ref{cor:lr_bound}. 

In the right panel of Figure \ref{fig:regret}, we numerically validate (2) by plotting
\begin{itemize}
	\item the expected numbers of arm pulls for both optimal and sub-optimal arms in dotted lines; 
	\item theoretical numbers of arm pulls $\{n_{a;T}^\ast\}$ in the solid lines.
\end{itemize}
The dotted lines closely align with the solid lines, and thereby verifying the conclusions in Propositions \ref{prop:E_n_aT} and \ref{prop:sub_n_aT}. For visualization purposes, we only display the curves for $\Delta \in [0.01,0.05]$. 

\subsection{Gaussian approximation of the empirical mean, and coverage of the CI's}

In the second simulation, we examine numerically (i) the Gaussian approximation of $\hat{\mu}_{a;T}$ in Theorem \ref{thm:mu_aT_CLT}, and (ii) the numerical performance of the coverage probabilities of the CI's defined in (\ref{def:CI_mean}). 

We choose the common exploration rate $\gamma_T=\sqrt{6 \log T/T}$ with time horizon $T=10,000$. The optimal arm has mean $\mu_1=1$ and the sub-optimal arm has mean $\mu_2=1-\sqrt{\log T/T}$. So the sub-optimality gap is $\Delta_2=\sqrt{\log T/T}$. 

\begin{figure}[t]
	\begin{minipage}[t]{0.4\textwidth}
		\includegraphics[width=\textwidth]{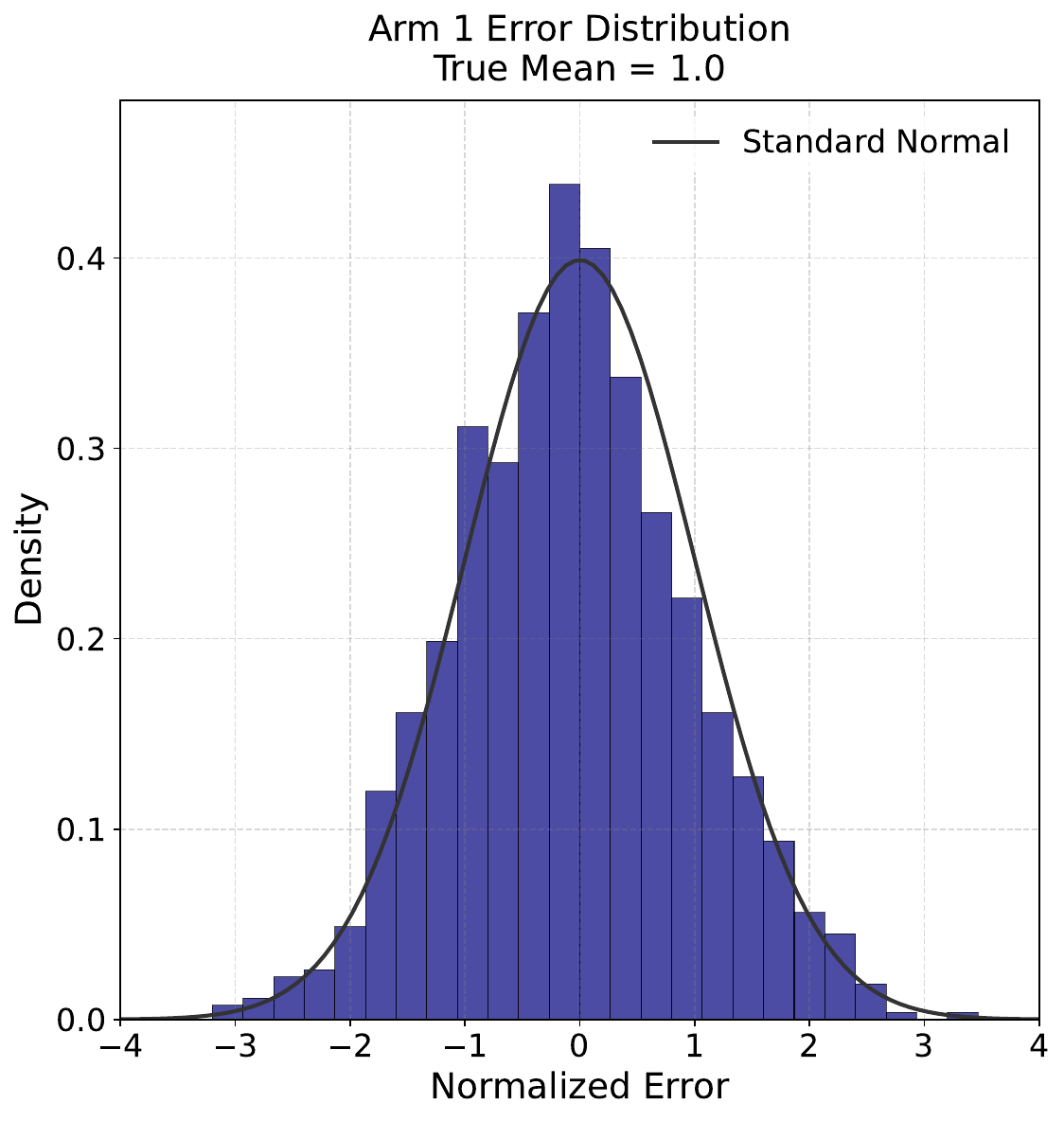}
	\end{minipage}
	\begin{minipage}[t]{0.415\textwidth}
		\includegraphics[width=\textwidth]{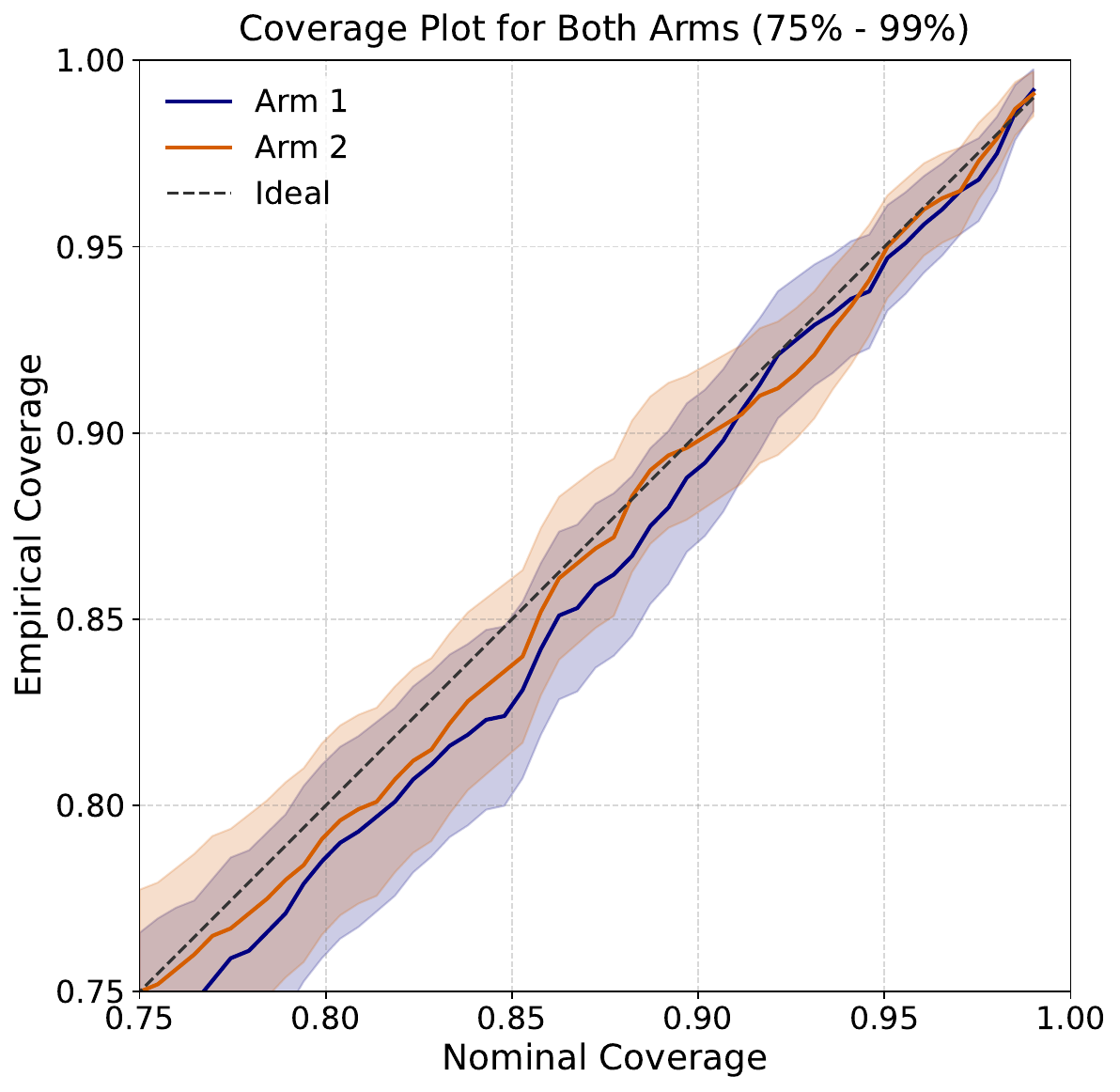}
	\end{minipage}
	\caption{ \emph{Left panel}: Gaussian approximation of the empirical mean for the optimal arm. \emph{Right panel}: Coverage probabilities for the CI's in (\ref{def:CI_mean})  for both arms.}
	\label{fig:CI}
\end{figure}

From the left panel, the Gaussian approximation of $\hat{\mu}_{a;T}$ (after normalization) appears quite accurate. From the right panel, coverage probabilities remain precise across various targeted nominal levels. These numerical findings are consistent for various choices of exploration rates $\gamma_T$'s and sub-optimality gaps $\Delta_2$'s.

\section{Proof preliminaries}\label{section:proof_preliminary}

We note that the allocation rule in Algorithm \ref{alg:ucb} can be written as 
\begin{align*}
A_{t} \in \argmax_{a \in [K]}\bigg\{ \frac{\hat{\mu}_{a;t-1} - \mu_\ast}{\sigma} 
+ \frac{\gamma_T}{  \sqrt{n_{a;t-1}}} \bigg\}, 
\end{align*}
which depends on the data and parameters only through the scaled $\mathcal{N}(0,1)$ noises $\xi_t$ and 
$(T,K,\Delta/\sigma,\gamma_T)$ due to location-scale invariance. 
Thus, $\E n_{a;T}$ and $R(\Theta)/\sigma$ are functions of $(T,K,\Delta/\sigma,\gamma_T)$ only. 
Therefore, we may take $\mu_\ast=0$ or $\sigma=1$ without loss of generality  throughout the proof.

We recall and define some further notation in the proofs below:
\begin{itemize}
	\item Let $n_\ast\equiv n_\ast(\Theta)$ be defined as in (\ref{def:n_ast}), and $n_{a;T}$ be defined as in (\ref{def:n_aT}).
	\item Let $\mathcal{A}_0\equiv \{a \in [K]: \Delta_a=0\}$ collect all optimal arms.
	\item Let $\mathcal{A}_+\equiv \{a \in [K]: \Delta_a>0\}$ collect all sub-optimal arms. 
\end{itemize}

\subsection{Time continuity of noiseless UCB}

The continuous-time, noiseless UCB and sample size growth curves in \eqref{eq:noiseless-1} for \texttt{UCB1} are given by
	\begin{align}\label{eq:noiseless}
	\mu^\ast_{+,t} = \mu_a + \sigma\gamma_T/(n^\ast_{a;t})^{1/2},\quad
	\sum_{a\in [K]}n^\ast_{a;t}=t.
	\end{align}
We may also write $\mu^\ast_{+,t} = \mu_\ast + \sigma\gamma_T/(n_{\ast,t})^{1/2}$ and
	\begin{align}\label{eq:lem:n^*_at}
	n^\ast_{a;t} \equiv n_{\ast,t}\big(1+n_{\ast,t}^{1/2}\Delta_a/(\sigma \gamma_T)\big)^{-2},\quad t\geq 0, 
	\end{align}
where $n_{\ast,t}$ is the solution of 
	\begin{align}\label{eq:lem:n_*t}
	\sum_{a\in [K]}n_{\ast,t}\big(1+n_{\ast,t}^{1/2}\Delta_a/(\sigma \gamma_T)\big)^{-2}=t,\quad t\geq 0. 
	\end{align}
Corresponding to the above, the noiseless regret growth curve is 
	\begin{align}\label{eq:noiseless-Reg}
	\texttt{Reg}^\ast_{t}(\Theta) \equiv \sum_{a\in [K]}\Delta_a n^\ast_{a;t}.
	\end{align}
We first study the smoothness of the sample size and regret growth curves. 

\begin{lemma}\label{lem:n^*_at}
	For $t\geq 0$, $n_{\ast,t}$ is the unique solution of \eqref{eq:lem:n_*t} 
	with $n_{\ast,t}\in [t/K,t/\abs{\mathcal{A}_0}]$. 
	The functions $n_{\ast,t}$ and $n^\ast_{1;t},\ldots,n^\ast_{K;t}$ are all 
	nonnegative and strict increasing in $t$ in $[0,\infty)$ and take value 0 
	at $t=0$. 
	In particular, $n_\ast = n_{\ast,T}$ fulfills equation (\ref{def:n_ast}). 
	Moreover, the following hold 
	with $D_\ast = \sum_{a\in [K]}\big(n^\ast_{a;T}/n_\ast\big)^{1/2}n^\ast_{a;T}/T$ defined in (\ref{def:D_ast}):
	\begin{align}\label{eq:lem:n^*_at-bd-1}
	& 1 - \frac{n^\ast_{a;t}}{n^\ast_{a;T}} \leq
	\min\bigg\{\frac{1}{D_\ast}, \frac{2\big(n^\ast_{a;t}/n_{\ast,t}\big)^{1/2}}{D_\ast}\bigg\}
	\bigg(1-\frac{t}{T}\bigg)_+,\ \ \forall a\in [K], 
	\\ \label{eq:lem:n^*_at-bd-2} 
	& 1 - \frac{n^\ast_{a;T}}{n^\ast_{a;t}} 
	\leq\frac{\big(n^\ast_{a;T}/n_\ast\big)^{1/2}}{D_\ast} 
	\bigg(\frac{t}{T}-1\bigg)_+,\ \ \forall a\in [K], 
	\\ \label{eq:lem:n^*_at-bd-3} 
	&	1\geq D_\ast \geq \max\big\{\abs{\mathcal{A}_0}n_\ast/T, \sqrt{T/(n_\ast K)}\big\}
	\geq  (\abs{\mathcal{A}_0}/K)^{1/3}, 
	\end{align}
	and  
	\begin{align}\label{eq:lem:n^*_at-bd-4}
	\left|\frac{\texttt{Reg}^\ast_{t}(\Theta)}{\texttt{Reg}^\ast_{T}(\Theta)}-1\right| \leq\bigg|\frac{t}{T}-1\bigg|. 
	\end{align}
\end{lemma}

\begin{proof}
\noindent (1). The stated properties of $n_{\ast,t}$ follow directly from the strict monotonicity 
and continuity of the function $n\mapsto \sum_{a\in [K]}n(1+n^{1/2}\Delta_a/(\sigma \gamma_T))^{-2}$ 
on $[0,\infty)$. The stated properties of $n^\ast_{a;t}$ follow immediately. 

\noindent (2). Let $h(x) \equiv n_{\ast,xT}/n_\ast$ and $y_a \equiv n_\ast^{1/2}\Delta_a/(\sigma\gamma_T)$. 
By the definition (\ref{eq:lem:n_*t}) of $n_{\ast,t}$,  
\begin{align*}
\frac{xT}{n_\ast} 
= \sum_{a\in [K]}\frac{1}{(1/\sqrt{h(x)}+y_a)^2},\quad x>0. 
\end{align*}
Differentiating both sides above with respect to $x$, we find that 
	\begin{align}\label{eq:h'(x)}
	\frac{T}{n_\ast} = \sum_{a\in [K]}\frac{h'(x)}{(1 + \sqrt{h(x)}y_a)^3},\quad x>0. 
	\end{align}
Consequently, $
1/h'(1) = T^{-1}\sum_{a \in [K]}n_\ast^{-1/2} \{ n_\ast (1+y_a)^{-2}\}^{3/2} = D_\ast $. Moreover, using that (i) $h'(x)\leq h'(1)$ for $x\in (0,1]$, as $h(x)$ is increasing in $x$, (ii) $(\d/\d x)(-h^{-1/2}(x))\leq h'(1)/2$ for $x\geq 1$, as  $h'(x)/h^{3/2}(x)$ is decreasing in $x$,
\begin{align}\label{ineq:n^*_at_h}
\begin{cases}
1-h(x) \leq(1-x)/D_\ast,& x \in [0,1];\\
1-h^{-1/2}(x)\leq(x-1)/(2D_\ast),& x \geq 1.
\end{cases}
\end{align}

\noindent \emph{Proof of (\ref{eq:lem:n^*_at-bd-1}) and (\ref{eq:lem:n^*_at-bd-2})}. 
Let $h_a(x) \equiv n^\ast_{a;xT}/n^\ast_{a;T}$. 
Because $\sigma\gamma_T/(n^\ast_{a;t})^{1/2} + \mu_a$ 
takes the common value $\mu^\ast_{+,t}$ as the noiseless UCB, 
$\{(n^\ast_{a;t})^{-1/2}\}_{a \in [K]}$ is a constant shift of each other that does not depend on $t$. It then follows that  
	\begin{align}\label{ineq:pf-lem:n^*_at}
	\frac{1}{\sqrt{h_a(x)}}-1 
	= \bigg(\frac{n^\ast_{a;T}}{n_\ast}\bigg)^{1/2}\bigg(\frac{1}{\sqrt{h(x)}} -1\bigg). 
	\end{align}
We consider two regimes $x\geq 1$ and $x \in (0,1]$ for the above identity:
\begin{itemize}
	\item For $x\geq 1$, \eqref{ineq:pf-lem:n^*_at} implies $1\leq h_a(x)\leq h(x)$ 
	due to $n^\ast_{a;T}\leq n_\ast$, so that 
	\begin{align*}
	1 - \frac{1}{h_a(x)} = \bigg(\frac{n^\ast_{a;T}}{n_\ast}\bigg)^{1/2}
	\bigg(1 + \frac{1}{\sqrt{h_a(x)}}\bigg)\bigg(1 - \frac{1}{\sqrt{h(x)}}\bigg)
	\leq\bigg(\frac{n^\ast_{a;T}}{n_\ast}\bigg)^{1/2}\frac{x-1}{D_\ast}. 
	\end{align*}
	In the last step above we used the second inequality (\ref{ineq:n^*_at_h}) and the trivial estimate $1+h_a^{-1/2}(x)\leq 2$. This proves \eqref{eq:lem:n^*_at-bd-2}. 
	\item For $x\in (0,1]$, \eqref{ineq:pf-lem:n^*_at} implies $h(x)\leq h_a(x)\leq 1$ and 
	\begin{align*}
	\frac{1-h_a(x)}{1-h(x)} &= \bigg(\frac{n^\ast_{a;xT}}{n_{\ast,xT}}\bigg)^{1/2}
	\frac{1+\sqrt{h_a(x)}}{1+\sqrt{h(x)}}
	\leq\min\bigg\{1,2\bigg(\frac{n^\ast_{a;xT}}{n_{\ast,xT}}\bigg)^{1/2}\bigg\}. 
	\end{align*}
	Now \eqref{eq:lem:n^*_at-bd-1} follows by using the first inequality in (\ref{ineq:n^*_at_h}).  
\end{itemize}

\noindent \emph{Proof of (\ref{eq:lem:n^*_at-bd-3})}. As $n^\ast_{a;T}\leq n_\ast$, \eqref{eq:lem:n^*_at} and \eqref{eq:lem:n_*t} imply $D_\ast\leq T^{-1}\sum_{a \in [K]} n_{a;T}^\ast =1$. 
As $n^\ast_{a;T}=n_\ast$ for $a\in\mathcal{A}_0$, $D_\ast \geq |\mathcal{A}_0|n_\ast/T$ by restricting the sum to $\mathcal{A}_0$. Furthermore, by Jensen's inequality applied to the variable $X \sim K^{-1}\sum_{a \in [K]} \delta_{n^\ast_{a;T}/n_\ast }$ and the convex function $x\mapsto x^{3/2}$ on $[0,\infty)$, 
\begin{align*}
D_\ast 
= \frac{\sum_{a\in [K]}\big(n^\ast_{a;T}/n_\ast\big)^{3/2}/K}{\sum_{a\in [K]}\big(n^\ast_{a;T}/n_\ast\big)/K} = \frac{\E X^{3/2}}{\E X}
\geq (\E X)^{1/2} = \bigg(\sum_{a\in [K]}\frac{n^\ast_{a;T}}{n_\ast K}\bigg)^{1/2}=\sqrt{ \frac{T}{n_\ast K}}. 
\end{align*}
Consequently, 
\eqref{eq:lem:n^*_at-bd-3} follows from 
 $\min_{x>0} \max\big(x/\sqrt{K},|\mathcal{A}_0|/x^2\big) = (|\mathcal{A}_0|/K)^{1/3}$. 

\noindent \emph{Proof of (\ref{eq:lem:n^*_at-bd-4})}. Let $H_{x}$ be the distribution putting mass $(1/\sqrt{h(x)}+y_a)^{-2}n_\ast/(xT)$ at $y_a = n_\ast^{1/2}\Delta_a/(\sigma\gamma_T)$ for all $a \in [K]$. 
By \eqref{eq:h'(x)}, 
\begin{align}\label{ineq:n^*_at_1}
\frac{xh'(x)}{h^{3/2}(x)}\int \frac{1}{1/\sqrt{h(x)} + y}H_{x}(\d{y}) =1. 
\end{align} 
To study the regret growth curve, let $f(x) \equiv \texttt{Reg}^\ast_{xT}(\Theta)/\texttt{Reg}^\ast_{T}(\Theta)$. 
By (\ref{eq:lem:n^*_at}) and (\ref{eq:noiseless-Reg}), we have $
\texttt{Reg}^\ast_{xT}(\Theta)= (\sigma \gamma_T n_\ast^{1/2}) \sum_{a \in [K]} y_a (1/\sqrt{h(x)}+y_a)^{-2}$,
so 
\begin{align*}
f(x) 
= \frac{\sum_{a\in[K]} y_a/(1/\sqrt{h(x)}+y_a)^2}{\sum_{a\in[K]} y_a/(1+y_a)^2}. 
\end{align*}
It follows that 
	\begin{align*}
	\frac{f'(x)}{f(x)} &= \frac{h'(x)\sum_{a\in[K]} y_a/(1+\sqrt{h(x)}y_a)^3}
		{\sum_{a\in[K]} y_a/(1/\sqrt{h(x)}+y_a)^2}= \frac{h'(x)\int y(1/\sqrt{h(x)}+y)^{-1}H_{x}(\d{y})}
		{h^{3/2}(x)\int y H_{x}(\d{y})}\\
	 &\stackrel{(\ast)}{=}  \frac{\int y(1/\sqrt{h(x)}+y)^{-1}H_{x}(\d{y})}
		{x\int y H_{x}(\d{y})\int (1/\sqrt{h(x)}+y)^{-1}H_{x}(\d{y})}\stackrel{(\ast\ast)}{\leq} 1/x. 
	\end{align*}
Here in $(\ast)$ we used (\ref{ineq:n^*_at_1}), and in $(\ast\ast)$ we used Chebyshev's correlation inequality, cf. \cite[Theorem 2.14]{boucheron2013concentration}.
Thus, as $f(1)=1$, $\log f(x) \leq\log x$ for $x>1$ and $\log f(x)\geq \log x$ for $x\in (0,1]$. 
This yields \eqref{eq:lem:n^*_at-bd-4}.  
\end{proof}

\subsection{Comparison inequality}
We rewrite the reward sequence as  
	\begin{align}\label{def:R_a i}
	R_{a,i} \equiv \mu_a + \sigma\xi_{a;i},\quad a \in [K]
	\end{align}
with $\{\xi_{a;i}\}_{a \in [K], i \in [T]}\stackrel{\mathrm{i.i.d.}}{\sim} \mathcal{N}(0,1)$, so that  \eqref{eqn:bandit-model} 
holds with $R_t = R_{A_t,n_{A_t;t}}$. Define  
	\begin{align}\label{def:W_a}
	W_a \equiv \max_{t \in [T] }\bigg|\frac{1}{t^{1/2}}\sum_{i=1}^t\xi_{a;i}\bigg|. 
	\end{align}
We note that $\{W_a\}_{a \in [K]}$ are i.i.d. 1-Lipschitz functions of standard Gaussian variables $\{\xi_{a;[T]}\}_{a \in [K]}$ with $W_a \approx \sqrt{2\log\log T}$ by the law of the iterated logarithm.

Let us define the events:
\begin{itemize}
	\item $E_0(\gamma_T)\equiv \big\{\max_{a \in [K]} \big( (\gamma_T-W_a)/n_{a;T}^{1/2}-\Delta_a/\sigma \big)>0 \big\}$.
	\item For $t>0$, $E_+(t,\gamma_T)\equiv \big\{\sum_{a\in [K]} n^\ast_{a;t}(1-W_a/\gamma_T)_+^2> T\big\}$.
	\item For $t>0$, $E_-(t,\gamma_T)\equiv \big\{\sum_{a\in [K]} n^\ast_{a;t}(1+W_a/\gamma_T)^2+K< T\big\}$.
\end{itemize}
Clearly $t\mapsto E_+(t,\gamma_T)$ is non-decreasing and $t\mapsto E_-(t,\gamma_T)$ is non-increasing.

\begin{lemma}\label{lem:n_aT_compare} 
Fix $T_+,T_->0$. Then the following hold.
\begin{enumerate}
	\item $\cup_{t>0} E_+(t,\gamma_T)  \subset E_0(\gamma_T)$.
	\item On $E_+(T_+,\gamma_T)$, $n_{a;T}
	\leq n^\ast_{a;T_+}(1+W_a/\gamma_T)^2+1$ holds for all $a \in [K]$.
	\item On $E_-(T_-,\gamma_T)\cap E_0(\gamma_T)$, $n_{a;T}\geq n^\ast_{a;T_-}(1-W_a/\gamma_T)_+^2
	$ holds for all $a \in [K]$.
\end{enumerate}
\end{lemma}	
\begin{proof}
Assume $\mu_\ast=0$ and $\sigma=1$ without loss of generality. 

\noindent (1). The noiseless UCB is positive, i.e., $\mu^\ast_{+,t} = \gamma_T/n_{\ast,t}^{1/2} > 0$ for all $t>0$. On the event $E_0^c(\gamma_T)$, 
\begin{align*}
\frac{\gamma_T-W_a}{ \sqrt{n_{a;T}}} - \Delta_a 
\leq0 < \mu^\ast_{+,t} = \frac{\gamma_T}{(n^\ast_{a;t})^{1/2}} - \Delta_a,\quad \forall a\in [K], \, t>0, 
\end{align*}
which implies $\sum_{a\in [K]} n^\ast_{a;t}(1-W_a/\gamma_T)_+^2 < T$ for any $t>0$. Thus, we have 
$E_0^c(\gamma_T)\subseteq \cap_{t>0} E_+^c(t,\gamma_T)$, 
equivalently  $\cup_{t>0} E_+(t,\gamma_T) \subseteq E_0(\gamma_T)$. 

\noindent (2)-(3). Let us now work on the event $E_0(\gamma_T)$. 
Let $T_a \in [1,T]$ be the last time arm $a$ is sampled. 
By the definition of the \texttt{UCB1} allocation rule 
and the definition of $W_a$ in \eqref{def:W_a}, we have the following basic inequality: 
When $n_{a;T}>1$, 
\begin{align*}
\max_{b\in [K]}\bigg(\frac{\gamma_T-W_b}{ \sqrt{n_{b;T_a-1}}} - \Delta_b\bigg)
\leq\max_{b\in [K]}\bigg(\hat{\mu}_{b;T_a-1}+\frac{\gamma_T}{ \sqrt{n_{b;T_a-1}}}\bigg)
\leq \hat{\mu}_{a;T_a-1}+ \frac{\gamma_T}{ \sqrt{n_{a;T_a-1} } }.
\end{align*}
Because the left-hand side is positive on $E_0(\gamma_T)$, and $n_{b;T_a-1}\leq n_{b;T}$, $n_{a;T_a-1}=n_{a;T}-1$ hold,  there exists certain positive $\overline{\mu}_{+}$ such that for all $(a,b) \in [K]^2$,
\begin{align*}
\frac{\gamma_T - W_b}{\sqrt{n_{b;T}}} - \Delta_b
\leq \overline{\mu}_{+} 
\leq\frac{\gamma_T + W_a}{ \sqrt{n_{a;T}-1}} - \Delta_a. 
\end{align*}
We note that the second inequality holds above automatically for $n_{a;T}=1$. 
Let $\hat{t}$ be the solution of $\mu^\ast_{+,\hat{t}} = \overline{\mu}_{+}$ 
as the noiseless UCB. The existence of such a random $\hat{t}$ is guaranteed as $\overline{\mu}_{+}>0$. 
By \eqref{eq:noiseless}, 
$\mu^\ast_{+,\hat{t}} = \gamma_T/(n^\ast_{a;\hat{t}})^{1/2} - \Delta_a$ for all $a \in [K]$, so that  
\begin{align*}
\frac{\gamma_T-W_a}{\sqrt{n_{a;T}}} 
\leq\frac{\gamma_T}{(n^\ast_{a;\hat{t}})^{1/2}} 
\leq\frac{\gamma_T+W_a}{ \sqrt{n_{a;T}-1}},\quad \forall a\in [K]. 
\end{align*}
It follows that 
\begin{align}\label{ineq:n_aT}
& n^\ast_{a;\hat{t}}(1-W_a/\gamma_T)_+^2
\leq n_{a;T}
\leq n^\ast_{a;\hat{t}}(1+W_a/\gamma_T)^2+1,\quad \forall a\in [K]. 
\end{align}
By summing over $a \in [K]$  and using the identity $\sum_{a \in [K]} n_{a;T}=T$, 
\begin{align}\label{ineq:n_aT-sum}
& \sum_{a\in [K]} n^\ast_{a;\hat{t}}(1-W_a/\gamma_T)_+^2
\leq T 
\leq\sum_{a\in [K]} n^\ast_{a;\hat{t}}(1+W_a/\gamma_T)^2+K.
\end{align}

On the event $E_+(T_+,\gamma_T)$, $E_0(\gamma_T)$ happens and the first inequality in \eqref{ineq:n_aT-sum} implies 
$\hat{t}\leq T_+$, so that the second inequality \eqref{ineq:n_aT} implies the claim in (2).  
On the event $E_0(\gamma_T)\cap E_-(T_-,\gamma_T)$, the second inequality of \eqref{ineq:n_aT-sum} implies 
$\hat{t} \geq T_-$, so that the first inequality \eqref{ineq:n_aT} implies the claim in (3). 
\end{proof}

\subsection{Some probabilistic results}

We first give two simple results.
\begin{lemma}\label{lem:Wa_small pert}
	There exists some universal constant $C>0$ such that if $\sqrt{\log \log T}/\gamma_T\leq 1/C$,
	\begin{align*}
	\max_{a \in [K]}\bigabs{\E(1\pm W_a/\gamma_T)_+^2 -1}\leq C\cdot \sqrt{\log \log T}/\gamma_T.
	\end{align*}
\end{lemma}
\begin{proof}
	Note that $\abs{(1\pm W_a/\gamma_T)_+^2 -1}\lesssim \abs{W_a}/\gamma_T+W_a^2/\gamma_T^2$, so we may use $ (\E \abs{W_a})^2\leq \E W_a^2\lesssim \log \log T$ (which can be easily obtained via integrating the tail in Lemma \ref{lem:BM_crossing}) to conclude. 
\end{proof}

\begin{lemma}\label{lem:E0_prob}
There exists some universal constant $C>0$ such that
\begin{align*}
\Prob\big(E_0^c(\gamma_T)\big)\leq C\log T\cdot \gamma_T e^{-\gamma_T^2/2}.
\end{align*}
\end{lemma}
\begin{proof}
We assume $\mu_\ast=0,\sigma=1$ for simplicity. Note that
\begin{align*}
\Prob\big(E_0^c(\gamma_T)\big)&=\Prob\big(\gamma_T-W_a<\sqrt{n_{a;T}}\cdot\Delta_a,\, \forall a \in [K]\big)\\
&\leq \Prob\big(\gamma_T-W_a<0,\, \forall a \in \mathcal{A}_0\big) \leq \mathfrak{p}_T(\gamma_T)\leq C\log T\cdot  \gamma_T e^{-\gamma_T^2/2},
\end{align*}
where in the last inequality we used Lemma \ref{lem:BM_crossing}.
\end{proof}

The next lemma provides a large deviation inequality for a general weighted sum of $\{(1\pm W_a/\gamma_T)_+^2\}_{a \in [K]}$. Its formulation is slightly involved mainly to accommodate all later proof needs. 
\begin{lemma}\label{lem:sum-prob-bd} 
	Fix $\omega \in \R_{\geq 0}^K$. There exists a universal constant $c_\ast>1$ such that the following hold. Take any $C_0>c_\ast$. Let  $\eta\geq \eta_T \equiv\sqrt{\log \log T}/\gamma_T$ satisfy $\eta\leq 1/C_0^2$. 
	\begin{enumerate}
		\item Let $T_+>0,\delta \in \R$ be such that $\big(\sum_{a \in [K]}  \omega_a n_{a;T_+}^\ast\big)/\big(\sum_{a \in [K]}  \omega_a n_{a;T}^\ast+\delta\big)_+\geq 1/(1-2C_0\eta)_+^2$. Then
		\begin{align*}
		\Prob\bigg(\sum_{a\in [K]} \omega_a n^\ast_{a;T_+}(1-W_a/\gamma_T)_+^2 \leq \bigg\{\sum_{a \in [K]}\omega_a n_{a;T}^\ast+\delta\bigg\}_+ \bigg) 
		\leq \exp\bigg(-\frac{C_0^2}{2}\frac{ \iprod{\omega}{n_{\cdot;T_+}^\ast }\gamma_T^2\eta^2}{\pnorm{\omega\circ n_{\cdot;T_+}^\ast}{\infty}}\bigg).
		\end{align*}
		\item Let $T_->0, \delta \in \R$ be such that $\big(\sum_{a \in [K]}  \omega_a n_{a;T_-}^\ast\big)/\big(\sum_{a \in [K]}  \omega_a n_{a;T}^\ast-\delta\big)_+\leq 1/(1+2C_0\eta)^2$. Then  
		\begin{align*}
		\Prob\bigg(\sum_{a\in [K]} \omega_a n^\ast_{a;T_-}(1+W_a/\gamma_T)_+^2 \geq  \bigg\{\sum_{a \in [K]}\omega_a n_{a;T}^\ast-\delta\bigg\}_+\bigg)
		\leq \exp\bigg(-\frac{C_0^2}{2}\frac{ \iprod{\omega}{n_{\cdot;T_-}^\ast }\gamma_T^2\eta^2}{\pnorm{\omega\circ n_{\cdot;T_-}^\ast}{\infty}}\bigg).
		\end{align*} 
	\end{enumerate}
\end{lemma}

\begin{proof}
We need a few more notation in the proof:
\begin{itemize}
	\item $\hat{T}_{\omega;\pm} \equiv \sum_{a\in [K]} \omega_a n^\ast_{a;T_\pm}(1\mp W_a/\gamma_T)_+^2$.
	\item $T_{\omega;\pm}\equiv \sum_{a \in [K]}  \omega_a n_{a;T_\pm}^\ast$, $T_{\omega}\equiv \sum_{a \in [K]}  \omega_a n_{a;T}^\ast$.
\end{itemize}
Note that the variables $\hat{T}_{\omega;\pm}\equiv \hat{T}_\pm(\xi_{[K];[T]})$ are functions of i.i.d. Gaussian variables $\{\xi_{a;t}: a \in [K], t \in [T]\}$, and therefore we shall identify $\hat{T}_{\omega;\pm}$ also as a map  $\R^{[K]\times [T]}\to \R$. We claim that 
\begin{align}\label{ineq:sum_prob_bd_1}
\bigpnorm{\hat{T}_{\omega;\pm}^{1/2}}{\mathrm{Lip}}\leq \pnorm{\omega\circ n_{\cdot;T_\pm}^\ast}{\infty}^{1/2} \gamma_T^{-1}\equiv \sigma_{\omega;\pm}.
\end{align}
Here for two vectors $a,b$, we write $a\circ b\equiv (a_ib_i)$. To prove (\ref{ineq:sum_prob_bd_1}), first note that $W_a=W_a(\xi_{a;[T]})$ so we may identify $W_a$ as a map $\R^{[T]}\to \R$, and with this identification, it is easy to verify $\pnorm{W_a}{\mathrm{Lip}}\leq 1$. Consequently, for any $u,v \in \R^{[K]\times [T]}$,
\begin{align*}
\abs{\hat{T}_{\omega;\pm}^{1/2}(u)-\hat{T}_{\omega;\pm}^{1/2}(v)  }&\leq \bigpnorm{ \Big( \omega_a^{1/2} n^{\ast,1/2}_{a;T_\pm} \gamma_T^{-1} \big(W_a(u_{a;[T]})-W_a(v_{a;[T]})\big) \Big)_{a \in [K]} }{}\\
&\leq \pnorm{\omega\circ n_{\cdot;T_\pm}^\ast}{\infty}^{1/2} \gamma_T^{-1}\cdot  \pnorm{u-v}{},
\end{align*}
proving the claim (\ref{ineq:sum_prob_bd_1}).

Using Gaussian-Poincar\'e inequality, we have 
\begin{align*}
\abs{ \E \hat{T}_{\omega;\pm} - (\E \hat{T}_{\omega;\pm}^{1/2})^2 }\leq T_{\omega;\pm}\gamma_T^{-2}.
\end{align*}
On the other hand, 
\begin{align*}
\abs{\E \hat{T}_{\omega;\pm} - T_{\omega;\pm} } &\leq \sum\nolimits_{a \in [K]} \omega_a n_{a;T\pm}^\ast\cdot \bigabs{\E[(1\pm W_a/\gamma_T)_+^2] -1} \leq  C\eta_T\cdot T_{\omega;\pm}.
\end{align*}
Consequently, as $\eta_T\leq 1/C_0^2$, 
\begin{align}\label{ineq:sum_prob_bd_2}
T_{\omega;\pm}^{1/2} \cdot (1-C_0\eta_T)_+\leq \E \hat{T}_{\omega;\pm}^{1/2}\leq T_{\omega;\pm}^{1/2}\cdot (1+C_0\eta_T).
\end{align}
Now using Gaussian concentration in the form of \cite[Theorem 2.5.7]{gine2015mathematical}, in view of (\ref{ineq:sum_prob_bd_1}) and (\ref{ineq:sum_prob_bd_2}), by choosing $\eta\geq \eta_T$ and $T_{\omega;+}$ such that $(T_\omega+\delta)_+^{1/2}\leq T_{\omega;+}^{1/2}(1-2C_0 \eta)_+$,
\begin{align*}
\Prob\big(\hat{T}_{\omega;+}^{1/2}\leq (T_\omega+\delta)_+^{1/2}\big)& \leq \Prob\Big(\hat{T}_{\omega;+}^{1/2}-\E \hat{T}_{\omega;+}^{1/2}\leq (T_{\omega}+\delta)_+^{1/2}- T_{\omega;+}^{1/2} \cdot (1-C_0\eta_T)_+ \Big) \\
&\leq \exp\big(-C_0^2 T_{\omega;+} \eta^2/(2\sigma_{\omega;+}^2)\big)\leq \exp\bigg(-\frac{C_0^2}{2}\frac{T_{\omega;+}\gamma_T^2\eta^2}{\pnorm{\omega\circ n_{\cdot;T_+}^\ast}{\infty}}\bigg).
\end{align*}
For the other direction, by choosing $\eta\geq \eta_T$ and $T_{\omega;-}$ such that $\big(T_\omega-\delta\big)_+^{1/2}\geq T_{\omega;-}^{1/2}(1+2C_0 \eta)$, 
\begin{align*}
\Prob\big(\hat{T}_{\omega;-}^{1/2}\geq \big(T_\omega-\delta\big)_+^{1/2}\big)& \leq \Prob\Big(\hat{T}_{\omega;-}^{1/2}-\E \hat{T}_{\omega;-}^{1/2}\geq \big(T_\omega-\delta\big)_+^{1/2}- T_{\omega;-}^{1/2} \cdot (1+C_0\eta_T) \Big) \\
&\leq \exp\big(-C_0^2 T_{\omega;-} \eta^2/(2\sigma_{\omega;-}^2)\big)\leq \exp\bigg(-\frac{C_0^2}{2}\frac{T_{\omega;-}\gamma_T^2\eta^2}{\pnorm{\omega\circ n_{\cdot;T_-}^\ast}{\infty}}\bigg).
\end{align*}
The proof is complete. 
\end{proof}

As a direct application of the proceeding concentration result, we have:

\begin{corollary}\label{cor:event_E_pm_prob}
Fix $C_0>2\vee c_\ast$ where $c_\ast>0$ comes from Lemma \ref{lem:sum-prob-bd}. Suppose
 $\epsilon_T\equiv \max\{\big(\sqrt{\log \log T}+\sqrt{\log \gamma_T}\big)/\gamma_T,K/T\}\leq 1/(20C_0^2)$. Then with $T_\pm\equiv (1\pm 20C_0\epsilon_T) T$, we have
\begin{align*}
\Prob\big(E_\pm^c(T_\pm,\gamma_T)\big)\leq \exp\bigg(-\frac{C_0^2}{2}\frac{T_\pm\log\gamma_T}{n_{\ast,T_\pm}}\bigg).
\end{align*}
\end{corollary}
\begin{proof} 
As $T_+/T=1+20C_0\epsilon_T\geq 1/(1-2C_0\epsilon_T)_+^2$, using Lemma \ref{lem:sum-prob-bd}-(1) with $\omega\equiv \bm{1}_K$, $\delta\equiv 0$ and $\eta\equiv \epsilon_T$,
\begin{align*}
\Prob\big(E_+^c(T_+,\gamma_T)\big)&\leq \Prob\bigg(\sum_{a\in [K]} n^\ast_{a;T_+}(1-W_a/\gamma_T)_+^2\leq  T\bigg)\leq \exp\bigg(-\frac{C_0^2}{2}\frac{T_+\log\gamma_T}{n_{\ast,T_+}}\bigg).
\end{align*}
The claim for $\Prob\big(E_-^c(T_-,\gamma_T)\big)$ holds similarly, where we take $\omega=\bm{1}_K$, $\delta=K$ and $\eta=\epsilon_T$ therein, and verify that $T_-/(T-K)\leq (1-20C_0\epsilon_T)/(1-\epsilon_T)\leq 1/(1+2C_0\epsilon_T)^2$.
\end{proof}

\section{Proofs for Section \ref{section:main_results}}\label{section:proof_regret_pseudo}

\subsection{In-probability controls: Proofs of Proposition \ref{prop:sub_n_aT} and Theorem \ref{thm:regret_pseudo}}

\begin{proof}[Proof of Proposition \ref{prop:sub_n_aT}]
	Fix some large universal constant $C_0>2$ and let $\epsilon_T\equiv  \max\{\big(\sqrt{\log \gamma_T}+\sqrt{\log \log T}\big)/\gamma_T,K/T\}\leq 1/C_0^3$. Let $T_\pm\equiv (1\pm C_0\epsilon_T)_+ T$. Consider the event $E_a(\gamma)\equiv \cap_{* \in \pm}E_*^c(T_\pm,\gamma_T)\cap E_0(\gamma_T)\cap \{W_a\leq \gamma\}$, where $\gamma\geq 1$. Then Lemma \ref{lem:E0_prob}, Corollary \ref{cor:event_E_pm_prob} and Lemma \ref{lem:BM_crossing} yield that
	\begin{align*}
	\Prob\big(E_a^c(\gamma)\big)/C&\leq  \gamma_T^{-2}+ \log T\cdot \gamma e^{-\gamma^2/2}.
	\end{align*}
	Using the comparison inequality in Lemma \ref{lem:n_aT_compare}, on the event $E_a(\gamma)$,
	\begin{align*}
	\frac{n_{a;T}}{n_{a;T}^\ast}\leq \big(1+\gamma/\gamma_T\big)^2\cdot \frac{n_{a;T_+}^\ast}{n_{a;T}^\ast}+ \frac{1}{ n_{a;T}^\ast}.
	\end{align*}  
	This means if furthermore $\gamma/\gamma_T\leq 1/C_0^3$ and $D_\ast^{-1}\epsilon_T\leq 1/C_0^3$, we may use Lemma \ref{lem:n^*_at} and the trivial estimate $1/n_{a;T}^\ast \leq (\sqrt{K/T}+\Delta_a/\gamma_T)^2$ to obtain 
	\begin{align*}
	\frac{n_{a;T}}{n_{a;T}^\ast}\leq 1+ C\cdot \bigg[\frac{1}{D_\ast}\bigg(\frac{\sqrt{\log \gamma_T}+\sqrt{\log \log T}}{\gamma_T}+\frac{K}{T}\bigg)+\frac{\Delta_a^2}{\gamma_T^2}+\frac{\gamma}{\gamma_T}\bigg].
	\end{align*}
	A similar lower bound holds. 
\end{proof}

\begin{proof}[Proof of Theorem \ref{thm:regret_pseudo}]
We start with the upper bound. Let $C_0>2$ be a large enough constant. Let us choose, similar as above, $\epsilon_T\equiv \max\{\big(\sqrt{\log \gamma_T}+\sqrt{\log \log T}\big)/\gamma_T,K/T\}$  and $T_+\equiv T(1+C_0\epsilon_T)$. Using the comparison inequality in Lemma \ref{lem:n_aT_compare}, on the event $E_+(T_+,\gamma_T)$,
\begin{align*}
\hat{R}^{\mathsf{p}}(\Theta)&\leq \sum_{a \in [K]}\Delta_a n^\ast_{a;T_+}(1+W_a/\gamma_T)^2+\sum_{a \in [K]}\Delta_a\equiv (I)+\sum_{a \in [K]}\Delta_a.
\end{align*}
With the choice $\delta_+\equiv -c_+ \epsilon_T \cdot \sum_{a \in [K]} \Delta_a n_{a;T}^\ast = -c_+ \epsilon_T \cdot \texttt{Reg}_T^\ast(\Theta)$ for some $c_+>0$ to be determined later, an application of (\ref{eq:lem:n^*_at-bd-4}) in Lemma \ref{lem:n^*_at} yields that 
\begin{align*}
\frac{\sum_{a \in [K]}  \Delta_a n_{a;T_+}^\ast}{\big(\sum_{a \in [K]}  \Delta_a n_{a;T}^\ast-\delta_+\big)_+}& =\frac{1}{(1+c_+ \epsilon_T)}\cdot \frac{ \texttt{Reg}_{T_+}^\ast(\Theta)  }{\texttt{Reg}_T^\ast(\Theta)}\leq \frac{1+C_0 \epsilon_T}{1+c_+ \epsilon_T}.
\end{align*}
Let us choose $c_+=10C_0$.  If $\epsilon_T\leq 1/(10C_0)$, the right hand side of the above display is bounded by $1-4C_0\epsilon_T\leq 1/(1+2C_0 \epsilon_T)^2$, so we may apply Lemma \ref{lem:sum-prob-bd}-(2) with $\omega=\Delta$ to conclude: on the event $E_+\cap E_+(T_+,\gamma_T)$ with $\Prob(E_+^c)\leq \exp\big(-C_0^2\gamma_T^2\epsilon_T^2/2\big)$, $
(I)\leq (1+c_+\epsilon_T) \cdot \texttt{Reg}_T^\ast(\Theta)$. The upper bound follows by using the trivial bound $\sum_{a \in [K]} \Delta_a/\texttt{Reg}_T^\ast(\Theta)\leq \max_{a \in [K]} (1/n_{a;T}^\ast)\leq (\sqrt{K/T}+\pnorm{\Delta}{\infty}/\gamma_T)^2$.

Next, for the lower bound, we choose similarly  $T_-\equiv T(1-C_0\epsilon_T)_+$. Using the other direction of the comparison inequality in Lemma \ref{lem:n_aT_compare}, on the event $E_-(T_-,\gamma_T)\cap E_0(\gamma_T)$,
\begin{align*}
\hat{R}^{\mathsf{p}}(\Theta)&\geq \sum_{a \in [K]}\Delta_a n^\ast_{a;T_-}(1-W_a/\gamma_T)_+^2.
\end{align*}
With the choice of $\delta_-\equiv -c_-\epsilon_T\cdot \texttt{Reg}_T^\ast(\Theta)$ for some $c_->0$ to be determined later, another application of (\ref{eq:lem:n^*_at-bd-4}) in Lemma \ref{lem:n^*_at} yields that
\begin{align*}
\frac{\sum_{a \in [K]}  \Delta_a n_{a;T_-}^\ast}{\big(\sum_{a \in [K]}  \Delta_a n_{a;T}^\ast+\delta\big)_+}& =\frac{1}{(1-c_- \epsilon_T)_+}\cdot \frac{ \texttt{Reg}_{T_-}^\ast(\Theta)  }{\texttt{Reg}_T^\ast(\Theta)}\geq \frac{(1-C_0 \epsilon_T)_+}{(1-c_-\epsilon_T)_+}.
\end{align*}
Choosing $c_- = 20C_0$. Then for $\epsilon_T\leq 1/(40C_0)$, the right hand side above is bounded from below by $(1-C_0\epsilon_T)(1+10 C_0\epsilon_T)\geq 1+8C_0 \epsilon_T\geq 1/(1-2C_0\epsilon_T)^2$. So we may apply Lemma \ref{lem:sum-prob-bd}-(1) with $\omega=\Delta$ to conclude.

Finally, the probability estimate is at most an absolute constant multiple of $\gamma_T^{-100}+\log T\cdot \gamma_T e^{-\gamma_T^2/2}\lesssim \gamma_T^{-100}$, as $\gamma_T\geq 80\sqrt{\log \log T}$ by the condition on $\epsilon_T$.
\end{proof}

\subsection{In-expectation controls: Proofs of Proposition \ref{prop:E_n_aT} and Theorem \ref{thm:regret_gaussian}}

The basic formulae we need for expectation control is the following.

\begin{lemma}\label{lem:regret} 
	Fix any $\omega \in \R_{\geq 0}^K$. Then for $T_+>0$, with $\kappa_a\equiv \kappa_a(T_+) \equiv\E(n_{a;T} -n^\ast_{a;T_+}-1)_+ \bm{1}_{E_+^c(T_+,\gamma_T)}$ for $a \in [K]$,
	\begin{align}\label{ineq:regret-1}
	\frac{\sum_{a \in [K]} \omega_a \E n_{a;T} }{\sum_{a \in [K]} \omega_a n_{a;T_+}^\ast }
	\leq \max_{a \in [K]}\E[(1+ W_a/\gamma_T)_+^2] 
	+ 	\frac{\sum_{a\in [K]}\omega_a (1+\kappa_a)}{\sum_{a \in [K]} \omega_a n_{a;T_+}^\ast}.
	\end{align}
	Moreover, for $T_->0$,
	\begin{align}\label{ineq:regret-2}
	\frac{ \sum_{a \in [K]} \omega_a \E n_{a;T} }{\sum_{a \in [K]} \omega_a n_{a;T_-}^\ast }
	\geq \min_{a \in [K]}\E[(1- W_a/\gamma_T)_+^2] - \Prob\{E^c_-(T_-,\gamma_T)\}-\Prob(E_0^c(\gamma_T)). 
	\end{align}
\end{lemma}

\begin{proof} 
	\noindent (1). First, using Lemma \ref{lem:n_aT_compare}, on the event $E_-(T_-,\gamma_T)\cap E_0(\gamma_T)$,
	\begin{align*}
	\sum_{a\in [K]} \omega_a n^\ast_{a;T_-}(1-W_a/\gamma_T)_+^2 \leq \sum_{a \in [K]} \omega_a n_{a;T}.
	\end{align*}
	Taking expectation on both sides of the inequality,
	\begin{align*}
	\sum_{a \in [K]} \omega_a \E n_{a;T}\geq \sum_{a \in [K]} \omega_a n_{a;T_-}^\ast\cdot \min_{a \in [K]} \E (1-W_a/\gamma_T)_+^2\bm{1}_{E_-(T_-,\gamma_T)\cap E_0(\gamma_T)}.
	\end{align*}
	The inequality \eqref{ineq:regret-2} follows. 
	
	\noindent (2). Next, on the event $E_+(T_+,\gamma_T)$, using again Lemma \ref{lem:n_aT_compare},
	\begin{align*}
	\sum_{a \in [K]} \omega_a n_{a;T}
	\leq \sum_{a\in [K]} \omega_a n^\ast_{a;T_+}(1+W_a/\gamma_T)^2+\sum_{a\in [K]}\omega_a. 
	\end{align*}
	Taking expectation on both sides of the inequality,
	\begin{align*}
	&\sum_{a \in [K]} \omega_a \E n_{a;T}\leq \sum_{a\in [K]} \omega_a n^\ast_{a;T_+}\E (1+W_a/\gamma_T)^2\big(1-\bm{1}_{E_+^c(T_+,\gamma_T)}\big)\\
	&\qquad +\sum_{a\in [K]}\omega_a\cdot \E\big(1-\bm{1}_{E_+^c(T_+,\gamma_T)}\big) +\sum_{a\in [K]} \omega_a \E n_{a;T} \bm{1}_{E_+^c(T_+,\gamma_T)}\\
	&\leq \sum_{a\in [K]} \omega_a n^\ast_{a;T_+}\E (1+W_a/\gamma_T)^2 + \sum_{a \in [K]} \omega_a (1+\kappa_a),
	\end{align*}
	proving (\ref{ineq:regret-1}).
\end{proof}

Let us now bound the key quantity $\kappa_a$ appearing in (\ref{ineq:regret-1}).

\begin{lemma}\label{lem:kappa_a}
There exists a universal constant $C_0>0$ such that if $\epsilon_T\equiv \max\{\big(\sqrt{\log \gamma_T}+\sqrt{\log \log T}\big)/\gamma_T,K/T\}\leq 1/C_0^3$, then with $T_+=T(1+C_0\epsilon_T)$,
\begin{align*}
\kappa_a\leq C_0\cdot \big(Te^{-\gamma_T^2/2}+1\big)\gamma_T^{-2}\cdot { n_{a;T_+}^\ast}.
\end{align*}
\end{lemma}
\begin{proof}
	Assume $\mu_\ast=0$ and $\sigma=1$ without loss of generality. Let 
	\begin{align*}
	S_1\equiv \{a\in [K]: \gamma_T^2/\Delta_a^2 \leq 4n^\ast_{a;T_+}\},\quad S_2 \equiv [K] \setminus S_1.
	\end{align*}
	Using the definition of $n_{a;T_+}^\ast$ via $n_{\ast,T_+}$ in (\ref{eq:lem:n^*_at}), equivalently we may represent $S_1, S_2$ as  $S_1=\{a\in [K]: \Delta_a/\gamma_T\geq n_{\ast,T_+}^{-1/2}\}$, $S_2=\{a\in [K]: \Delta_a/\gamma_T< n_{\ast,T_+}^{-1/2}\}$. Let $C_0>2$ be a large enough constant, and $\epsilon_T, T_+$ be as in the statement.
	
	\noindent (\textbf{Case 1}: $a \in S_1$). Fix an optimal arm $a^\ast \in [K]$, let $X_{a^*} \equiv \min_{t \in [T]}\big\{t^{-1}\sum_{i \in [t]} \xi_{a^\ast;i}+t^{-1/2}\gamma_T\big\}$. We define a one-sided version $W_{a^*}^-\equiv \big(\min_{t \in [T]} t^{-1/2}\sum_{i \in [t]} \xi_{a^\ast;i}\big)_+$. On $\{W_{a^\ast}^-\geq \gamma_T\}$, we have $X_{a^*}\leq 0$. Then
	\begin{align}\label{ineq:kappa_a_case1}
	\kappa_a&\leq \E \big(n_{a;T} -n^\ast_{a;T_+}-1\big)_+ \bm{1}_{E_+^c(T_+,\gamma_T)}\big(\bm{1}_{W_{a^\ast}^-\geq \gamma_T}+\bm{1}_{W_{a^\ast}^-<\gamma_T}\big)\nonumber\\
	&\leq \E (n_{a;T}-1)\bm{1}_{X_{a^\ast}\leq 0}+(\log \gamma_T)^2\cdot n_{a;T_+}^\ast \Prob\big(E_{+}^c(T_{+},\gamma_T) \big)\nonumber\\
	&\qquad + \E \big(n_{a;T} -(1+\log^2\gamma_T)n^\ast_{a;T_+}-1\big)_+\bm{1}_{ \{W_{a^\ast}^-<\gamma_T\} }\nonumber\\
	&\equiv I_1+I_2+I_3.
	\end{align}
	For $I_1$, note that on the event $\{-\Delta_a/2\leq X_{a^*} \leq 0\}$, 
	\begin{align*}
	n_{a;T}&\leq  1+\sum_{t \in [(K+1):T]} \bm{1}\,\bigg(\hat{\mu}_{a;t-1}+\frac{ \gamma_T}{\sqrt{n_{a;t-1}}}\geq \mu_\ast+ X_{a^*}\bigg)\\
	&\leq 1+ \sum_{t \in [T]} \bm{1}\,\bigg(\sum_{i \in [t]}\xi_{a;i}+\sqrt{t}\gamma_T\geq t\Delta_a/2\bigg).
	\end{align*}
	As $X_{a^*}$ is independent of $\{\xi_{a;[T]}\}$ for $a \in S_1$, with $\mathsf{Z}\sim \mathcal{N}(0,1)$,
	\begin{align*}
	I_1&\leq \E (n_{a;T}-1)\bm{1}_{-\Delta_a/2\leq X_{a^*} \leq 0}+ T\cdot \Prob\big(X_{a^*}<-\Delta_a/2\big)\\
	&\leq \Prob(X_{a^*}\leq 0)\sum_{t \in [T]} \Prob\bigg(\sum_{i \in [t]}\xi_{a;i}+\sqrt{t}\gamma_T\geq t\Delta_a/2\bigg)+ T\cdot \Prob\big(X_{a^*}<-\Delta_a/2\big)\\
	&\leq \mathfrak{p}_T(\gamma_T) \sum_{t \in [T]} \Prob\big( 4(\mathsf{Z}+\gamma_T)_+^2/\Delta_a^2\geq t\big)+ T\cdot \mathfrak{q}_T(\gamma_T,\Delta_a/2)\\
	&\leq C \log T\cdot \mathfrak{p}_T(\gamma_T)\cdot \gamma_T^2\cdot \Delta_a^{-2}+T\cdot \mathfrak{q}_T(\gamma_T,\Delta_a/2).
	\end{align*}
	Now using Lemma \ref{lem:BM_crossing}, $\mathfrak{p}_T(\gamma_T)\lesssim \log T\cdot \gamma_T e^{-\gamma_T^2/2}$ and 
	\begin{align*}
	&T\cdot \mathfrak{q}_T(\gamma_T,\Delta_a/2)\lesssim \frac{T}{\Delta_a^2} \int_{\gamma_T}^\infty (z-\gamma_T)^2 e^{-z^2/2}\,\d{z}\\
	&= \Delta_a^{-2}\cdot e^{\log T-\gamma_T^2/2}\int_0^\infty u^2 e^{-u^2/2-\gamma_T u}\,\d{u}\lesssim \Delta_a^{-2}\cdot e^{\log T-\gamma_T^2/2}.
	\end{align*}
	Consequently, combining the above estimates and using that $1/\Delta_a^2\leq 4n_{a;T_+}^\ast/\gamma_T^2$ for $a \in S_1$, we arrive at the estimate: for some numerical constant $C>0$,
	\begin{align}\label{ineq:kappa_a_case1_1}
	I_1\leq C \big(\log^2 T\cdot \gamma_T^3 e^{-\gamma_T^2/2}+ e^{\log T-\gamma_T^2/2}\big) \gamma_T^{-2}\cdot n_{a;T_+}^\ast.
	\end{align}	
	For $I_2$, using Corollary \ref{cor:event_E_pm_prob}, we have 
	\begin{align}\label{ineq:kappa_a_case1_2}
	I_2&\leq  n_{a;T_+}^\ast \log^2\gamma_T \cdot e^{-(C_0^2/2)\log \gamma_T}\leq \gamma_T^{-2}\cdot  n_{a;T_+}^\ast .
	\end{align}
	For $I_3$, note that on the event $\{n_{a;T}=t\}\cap \{W_{a^\ast}^-<\gamma_T\}$, we have $t^{-1}\sum_{i \in [t]} \xi_{a;i}-\Delta_a+\gamma_T/t^{1/2}>0$, it follows that
	\begin{align*}
	I_3&= \E \sum_{(1+\log^2\gamma_T)n_{a;T_+}^\ast+1\leq t\leq T}\bm{1}_{n_{a;T}=t} \big(n_{a;T} -(1+\log^2\gamma_T)n^\ast_{a;T_+}-1\big)_+\bm{1}_{ \{W_{a^\ast}^-<\gamma_T\} }\\
	&\leq T  \sum_{(1+\log^2\gamma_T)n_{a;T_+}^\ast+1\leq t\leq T} \Prob\bigg(\frac{1}{t}\sum_{i=1}^t\xi_{a;i} - \Delta_a + \frac{\gamma_T}{t^{1/2}} > 0\bigg).
	\end{align*}
	Note that for $a \in S_1$ and any $c_0\geq 1$, by enlarging $C_0>0$ if necessary to ensure $\log \gamma_T\geq 2(c_0+1)$, $(1+\log^2\gamma_T)n_{a;T_+}^\ast+1\geq (c_0+1)^2 \gamma_T^2/\Delta_a^2+1$, and therefore we may continue bounding the above display by
	\begin{align}\label{ineq:kappa_a_case1_3}
	I_3&\leq T \sum_{ (c_0+1)^2\gamma_T^2/\Delta_a^2+1  \leq t\leq T} \Prob\big(\mathsf{Z}>\sqrt{t}\Delta_a-\gamma_T\big)\nonumber\\
	&\leq T \sum_{ (c_0+1)^2\gamma_T^2/\Delta_a^2+1  \leq t\leq T}e^{-\big(\frac{c_0}{c_0+1}\big)^2 \frac{\Delta_a^2}{2} t}\leq T\int_{(c_0+1)^2\gamma_T^2/\Delta_a^2}^\infty e^{- \big(\frac{c_0}{c_0+1}\big)^2\frac{\Delta_a^2}{2} t}\,\d{t}\nonumber\\
	& \lesssim \big(Te^{-(c_0^2/2)\gamma_T^2}\big)\cdot \Delta_a^{-2} \leq C \cdot \big(Te^{-(c_0^2/2)\gamma_T^2}\big) \gamma_T^{-2}\cdot n_{a;T_+}^\ast.
	\end{align}
	Now plugging in (\ref{ineq:kappa_a_case1}) the estimates (\ref{ineq:kappa_a_case1_1})-(\ref{ineq:kappa_a_case1_3}), by choosing $c_0=1$, there exists a universal constant $C_0>0$ such that for $\epsilon_T\leq 1/C_0^3$ and $T_+=T(1+C_0\epsilon_T)$, we have for all $a \in S_1$,
	\begin{align*}
	\kappa_a\leq C \big(\log^2 T\cdot \gamma_T^3 e^{-\gamma_T^2/2}+ e^{\log T-\gamma_T^2/2}+1\big) \gamma_T^{-2}\cdot n_{a;T_+}^\ast.
	\end{align*}
	The term $\log^2 T\cdot \gamma_T^3 e^{-\gamma_T^2/2}$ can be assimilated into $1$.
	
	\noindent (\textbf{Case 2}: $a \in S_2$). For $a \in S_2$, using the trivial bound $n_{a;T}\leq T$, and $n_{\ast,T_+}\leq 4n_{a;T_+}^\ast$ due to $a \in S_2$, and Corollary \ref{cor:event_E_pm_prob}, for $C_0\geq 2$,
	\begin{align*}
	\kappa_a&\leq T\cdot \Prob\big(E_+^c(T_+,\gamma_T)\big)\leq 4n_{a;T_+}^\ast\cdot \frac{T}{n_{\ast,T_+} }\exp\bigg(-\frac{T_+\cdot 2\log \gamma_T}{n_{\ast,T_+}}\bigg)\leq   4\gamma_T^{-2}\cdot n_{a;T_+}^\ast,
	\end{align*} 
	where the last inequality follows as the map $x\mapsto x e^{-x u}, x\geq 1$ attains maximum at $x=1$ for $u\geq 1$.
\end{proof}

\begin{proof}[Proof of Proposition \ref{prop:E_n_aT}]
We divide the proof into two steps.

\noindent (\textbf{Step 1}). We shall first prove that
\begin{align}\label{ineq:proof_E_n_aT_1}
\max_{a \in [K]} \bigabs{\E n_{a;T}/n_{a;T}^\ast-1}\leq C\cdot \big(D_\ast^{-1}\err(\Theta)+\vartheta_T^\ast\big).
\end{align}	
For any $a \in [K]$, let $\omega\equiv \bm{1}_{\cdot =a}$. Now we apply Lemma \ref{lem:regret}-(1) along with the estimates in Lemma \ref{lem:Wa_small pert} and Lemma \ref{lem:kappa_a}, there exists a  universal constant $C_0>10$ such that if $\epsilon_T\equiv \max\{\big(\sqrt{\log \gamma_T}+\sqrt{\log \log T }\big)/\gamma_T,K/T\}\leq 1/C_0^3$, then with $T_+\equiv T(1+C_0\epsilon_T)$,
\begin{align*}
\E n_{a;T}
\leq \bigg[1+C\cdot \bigg(\frac{\sqrt{\log \log T}}{\gamma_T}
+ 	\frac{T e^{-\gamma_T^2/2}}{\gamma_T^2}\bigg)\bigg]\cdot n_{a;T_+}^\ast+1.
\end{align*}
Now using (\ref{eq:lem:n^*_at-bd-2}) in Lemma \ref{lem:n^*_at}, if $D_\ast^{-1}\epsilon_T\leq 1/C_0^2$, we have $n_{a;T_+}^\ast \leq n_{a;T}^\ast \big(1+2C_0D_\ast^{-1}\epsilon_T\big)$. By further using the trivial estimate $1/n_{a;T}^\ast \leq (\sqrt{K/T}+\Delta_a/\gamma_T)^2$, combined with the above display, we obtain the upper bound of (\ref{ineq:proof_E_n_aT_1}):
\begin{align*}
\bigg(\frac{\E n_{a;T}}{n_{a;T}^\ast  }-1\bigg)_+/C
\leq \frac{1}{D_\ast}\bigg(\frac{\sqrt{\log \gamma_T}+\sqrt{\log \log T}}{\gamma_T} + \frac{K}{T} \bigg)
+ 	\frac{T e^{-\gamma_T^2/2}}{\gamma_T^2} +\frac{\Delta_a^2}{\gamma_T^2}.
\end{align*}
The lower bound of (\ref{ineq:proof_E_n_aT_1}) is easier, as with $T_-\equiv T(1-C_0\epsilon_T)_+$, we may directly estimate from Lemma \ref{lem:regret}-(2) and the probabilistic Lemma \ref{lem:Wa_small pert}, Lemma \ref{lem:E0_prob} and Corollary \ref{cor:event_E_pm_prob} that
\begin{align*}
\E n_{a;T}\geq \Big(1-C\cdot \big(\sqrt{\log \log T}/\gamma_T+\log T\cdot \gamma_T e^{-\gamma_T^2/2}\big) \Big)_+\cdot n_{a;T_-}^\ast. 
\end{align*}
The term $\log T\cdot \gamma_T e^{-\gamma_T^2/2}$ can be assimilated into $\sqrt{\log \log T}/\gamma_T$, as $\gamma_T\geq 10^3 \sqrt{\log \log T}$. The desired lower bound now follows by using (\ref{eq:lem:n^*_at-bd-1}) in Lemma \ref{lem:n^*_at} to produce a lower bound $n_{a;T_-}^\ast\geq n_{a;T}^\ast \big(1-C_0 D_\ast^{-1}\epsilon_T\big)_+$. This completes the proof for (\ref{ineq:proof_E_n_aT_1}).

\noindent (\textbf{Step 2}). In this step, we strength (\ref{ineq:proof_E_n_aT_1}) to the stronger control $\E \abs{n_{a;T}/n_{a;T}^\ast-1}$. To this end, note that
\begin{align}\label{ineq:proof_E_n_aT_2}
\E \abs{n_{a;T}-n_{a;T}^\ast}&=\E (n_{a;T}-n_{a;T}^\ast)\bm{1}_{n_{a;T}\geq n_{a;T}^\ast }+\E (n_{a;T}^\ast-n_{a;T})\bm{1}_{n_{a;T}^\ast > n_{a;T}}\nonumber\\
&\leq \abs{\E (n_{a;T}-n_{a;T}^\ast)}+ 2\E (n_{a;T}^\ast-n_{a;T})\bm{1}_{n_{a;T}^\ast > n_{a;T}}.
\end{align}
So we only need to deal with the second term on the right hand side of the above display. On the event $E_-(T_-,\gamma_T)$, using the comparison inequality in Lemma \ref{lem:n_aT_compare}, on the event $E_-(T_-,\gamma_T)\cap E_0(\gamma_T)$, 
\begin{align*}
\big(n_{a;T}^\ast-n_{a;T}\big) \bm{1}_{n_{a;T}^\ast >  n_{a;T}}&\leq \Big( n_{a;T}^\ast- n^\ast_{a;T_-}(1-W_a/\gamma_T)_+^2\Big) \bm{1}_{n_{a;T}^\ast > n_{a;T}}\\
&\leq n_{a;T}^\ast\cdot \big(1- (1-W_a/\gamma_T)_+^2\big)+\abs{ n_{a;T}^\ast- n^\ast_{a;T_-}}.
\end{align*}
Consequently, taking expectation on the prescribed event and using Lemma \ref{lem:n^*_at} and Lemma \ref{lem:Wa_small pert},
\begin{align*}
&\E\big(n_{a;T}^\ast-n_{a;T}\big) \bm{1}_{n_{a;T}^\ast >  n_{a;T}} \bm{1}_{ E_-(T_-,\gamma_T)\cap E_0(\gamma_T) }\\
&\leq n_{a;T}^\ast\cdot \big(1- \E (1-W_a/\gamma_T)_+^2\big) +\abs{ n_{a;T}^\ast- n^\ast_{a;T_-}}\\
&\leq C n_{a;T}^\ast\cdot \big(\sqrt{\log \log T}/\gamma_T+ D_\ast^{-1}\epsilon_T\big).
\end{align*}
On the other hand, using Lemma \ref{lem:E0_prob} and Corollary \ref{cor:event_E_pm_prob},
\begin{align*}
&\E\big(n_{a;T}^\ast-n_{a;T}\big) \bm{1}_{n_{a;T}^\ast >  n_{a;T}} \big(1-\bm{1}_{ E_-(T_-,\gamma_T)\cap E_0(\gamma_T) }\big)\\
&\leq n_{a;T}^\ast\cdot \big(\Prob(E_-^c(T_-,\gamma_T))+\Prob(E_0^c(\gamma_T))\big)\\
&\leq C n_{a;T}^\ast\cdot \big(\log T\cdot \gamma_T e^{-\gamma_T^2/2}+\gamma_T^{-100}\big).
\end{align*}
Combining the above two displays, we have
\begin{align}\label{ineq:proof_E_n_aT_3}
\E \big(n_{a;T}^\ast-n_{a;T}\big) \bm{1}_{n_{a;T}^\ast >  n_{a;T}}\leq C\cdot \big(D_\ast^{-1}\err(\Theta)+\vartheta_T^\ast\big).
\end{align}
The claim now follows from estimating the two terms in (\ref{ineq:proof_E_n_aT_2}) by (\ref{ineq:proof_E_n_aT_1}) in Step 1 and (\ref{ineq:proof_E_n_aT_3}) above.
\end{proof}

\begin{proof}[Proof of Theorem \ref{thm:regret_gaussian}]
The proof is essentially the same as Step 1 above, but now Lemma \ref{lem:regret} is applied with $\omega\equiv \Delta$, and Lemma \ref{lem:n^*_at} is used with (\ref{eq:lem:n^*_at-bd-4}) that is free of $D_\ast$, with the trivial bound $\sum_{a \in [K]} \Delta_a/\texttt{Reg}_T^\ast(\Theta)\leq \max_{a \in [K]} (1/n_{a;T}^\ast)\leq (\sqrt{K/T}+\pnorm{\Delta}{\infty}/\gamma_T)^2$.
\end{proof}

\subsection{Proof of Corollary \ref{cor:lr_bound}}
We assume $\sigma=1$. Let $r_{\Theta}\equiv \texttt{Reg}_T^\ast(\Theta)/R^{\mathsf{LR}}(\Theta)$. Then we have an easy (trivial) upper bound $r_{\Theta}\leq 1$. 

\noindent (1). Using the apriori estimate $n_\ast \geq T/K$,
\begin{align*}
 r_{\Theta}\geq \frac{\sum_{a: \Delta_a>0 } \Delta_a \big(\sqrt{K/T}+{\Delta_a}/{\gamma_T}\big)^{-2}}{\sum_{a: \Delta_a>0} {\gamma_T^2}/{\Delta_a}}\geq \min_{a \in [K]}\bigg( \frac{ \gamma_T}{ \sqrt{T/K}\Delta_a} +1\bigg)^{-2}\geq \big(1- C/L\big)_+.
\end{align*}
Now using Theorem \ref{thm:regret_gaussian}, we have
\begin{align}\label{ineq:lr_bound_1}
\biggabs{\frac{R(\Theta)}{  R^{\mathsf{LR}}(\Theta) }-1  - \big(r_{\Theta}-1\big)}\leq C \cdot \err(\Theta) \cdot r_\Theta. 
\end{align}
Consequently, for any $L>C$, $
\abs{ {R(\Theta)}/{  R^{\mathsf{LR}}(\Theta) }-1  }\lesssim {1}/{L}+\err(\Theta)$. 

\noindent (2). Let $L\equiv \epsilon \in (0,1/2)$ and $\Delta_1=0$, $\Delta|_{[2:K]}\equiv \epsilon \sqrt{2\log T/(T/K)}$. Let $n_\ast \equiv \rho (T/K)$ for some $\rho \in [1,K]$. Then as
\begin{align*}
T &= n_\ast+ (K-1)\cdot \big(1/\sqrt{n_\ast}+\epsilon\sqrt{K/T}\big)^{-2}\\
&=\rho(T/K)+(1-1/K)\cdot(\rho^{-1/2}+\epsilon)^{-2}\cdot T\geq (\rho^{-1/2}+\epsilon)^{-2}\cdot T/2,
\end{align*}
we have $1\leq \rho \leq (1/\sqrt{2}-\epsilon)^{-2}\leq 24$ for $\epsilon \in (0,1/2)$. Consequently,
\begin{align*}
r_{\Theta}\lesssim \frac{ \sum_{a \in [2:K]} \Delta_a n_\ast(\Theta)}{\log T\cdot \sum_{a \in [2:K]} 1/\Delta_a}\lesssim \epsilon^2.
\end{align*}
Moreover, with this choice of $\Delta$, we may compute $\err(\Theta)\asymp \sqrt{\log \log T/\log T}+K/T$. This means $
\inf_{ \Delta \in \mathscr{C}(\epsilon) } {R(\Theta)}/{  R^{\mathsf{LR}}(\Theta) } \leq C\cdot \epsilon^2$ by (\ref{ineq:lr_bound_1}).
 \qed

\subsection{Proof of Corollary \ref{cor:minimax_risk}}

We assume $\sigma=1$. We only need to choose a sub-class of $\Delta$ to provide a lower bound for the maximum risk. Let us choose $\Delta$ with $\Delta_1=0$ and $\Delta_2=\cdots=\Delta_K=\tau \sqrt{2\log T/(T/K)}$ for some $\tau>0$ to be tuned later. From equation (\ref{def:n_ast}), we have a simple estimate $n_\ast\geq T/K$. Using this lower bound, with some calculations we have
\begin{align*}
\sum_{a \in \mathcal{A}_+} \Delta_a \big(1/\sqrt{n_{\ast}}+{\Delta_a}/{\gamma_T}\big)^{-2}\geq (1-1/K)\sqrt{2 TK\log T}\cdot \frac{\tau}{(1+\tau)^2}. 
\end{align*}
This suggests the choice $\tau=1$ to achieve the maximum in the right hand side of the above display. The claim follows by computing for the error term $\err(\Theta)$. \qed

\section{Proofs for Section \ref{section:application}}\label{section:proof_applications}

\subsection{A quantitative martingale central limit theorem}

The proofs for both Theorems \ref{thm:mu_aT_CLT} and \ref{thm:ridge_CLT} rely on the following version of a quantitative martingale central limit theorem, adapted from \cite[Theorem 1.2]{mourrat2013rate} and originally due to \cite{joos1993nonuniform}.

\begin{proposition}\label{prop:quant_clt_mtg}
	Let $\{(Y_t,\mathcal{F}_t)\}_{t \in [T]}$ be a square-integrable martingale difference sequence. Let $
	s_{T}^2\equiv  \sum_{t \in [T]} \E Y_t^2$ and $V_{T}^2\equiv {s_T^{-2}}\sum_{t \in [T]} \E [Y_t^2|\mathcal{F}_{t-1}]$.
	Then for any pair $(p,q) \in [1,\infty)$, there exists some constant $C=C(p,q)>0$ such that
	\begin{align*}
	\mathfrak{d}_{ \mathrm{Kol} }\bigg( \frac{1}{s_T}\sum_{t \in [T]} Y_t, \mathsf{Z}\bigg)\leq C\cdot \bigg[\pnorm{V_T^2-1}{p}^{ \frac{p}{2p+1} }+\bigg(s_T^{-2q}\sum_{t \in [T]} \pnorm{Y_t}{2q}^{2q}\bigg)^{\frac{1}{2q+1}}\bigg].
	\end{align*}
	Here $\mathsf{Z}\sim \mathcal{N}(0,1)$ is independent of all other variables.
\end{proposition}

\subsection{Proof of Theorem \ref{thm:mu_aT_CLT}}

	Without loss of generality, we assume $\sigma=1$. 
	
	\noindent (\textbf{Step 1}). In this step, we show that for any $\theta \in (0,1/4)$, there exists some $C_\theta>0$ such that with $\Delta_{a;T}(\rho)\equiv \Prob\big(\abs{\big\{{\E n_{a;T}}/{ n_{a;T}}\big\}^{1/2}-1}>\rho\big)$,
	\begin{align}\label{ineq:be_clt_mu_aT_step1}
	&\mathfrak{d}_{ \mathrm{Kol} }\Big(n_{a;T}^{1/2}\big(\hat{\mu}_{a;T}-\mu_a\big),\mathsf{Z}\Big)/C_\theta\nonumber\\
	&\leq \E^{1/3} \biggabs{\frac{n_{a;T}}{\E n_{a;T}}-1} +\frac{1}{ \{\E n_{a;T}\}^{1/2-\theta }  }+ \inf_{\rho \in (0,1/2)}\big(\rho+\Delta_{a;T}(\rho)\big).
	\end{align}
	To prove (\ref{ineq:be_clt_mu_aT_step1}), we write $Z_T\equiv n_{a;T}^{1/2}\big(\hat{\mu}_{a;T}-\mu_a\big)$, and let
	\begin{align*}
	Y_t\equiv \{n_{a;T}^\ast\}^{-1/2}\bm{1}_{A_t=a}\cdot (R_t-\mu_a).
	\end{align*}
	Let $\mathcal{F}_t$ be the $\sigma$-field generated by $R_1,\ldots,R_t$. Then we may compute $s_T^2 \equiv\E n_{a;T}/n_{a;T}^\ast$, and 
	\begin{align*}
	V_T^2 \equiv \frac{1}{s_T^2} \sum_{t \in [T]} \E [Y_t^2|\mathcal{F}_{t-1}] =\frac{\sum_{t \in [T]} \bm{1}_{A_t=a}}{\E n_{a;T}}= \frac{n_{a;T}}{\E n_{a;T}}.
	\end{align*}
	Moreover, for any $q\geq 1$, with $M_q\equiv \E \mathsf{Z}^{2q}$, 
	\begin{align*}
	\frac{1}{s_T^{2q}}\sum_{t \in [T]} \pnorm{Y_t}{2q}^{2q}=\frac{\{n_{a;T}^\ast\}^{q}}{ \{\E n_{a;T}\}^q }\cdot \frac{1}{\{n_{a;T}^\ast\}^{q} } \sum_{t \in [T]} \E \bm{1}_{A_t = a}\cdot M_q = \frac{M_q}{ \{\E n_{a;T}\}^{q-1}  }.
	\end{align*}
	On the other hand,  
	\begin{align*}
	\frac{1}{s_T}\sum_{t \in [T]} Y_t &= \frac{1}{ \{\E n_{a;T}\}^{1/2} }\sum_{t \in [T]} \bm{1}_{A_t=a}\cdot (R_t-\mu_a)= \bigg(\frac{n_{a;T}}{\E n_{a;T}}\bigg)^{1/2} Z_T. 
	\end{align*}
	Consequently, using Proposition \ref{prop:quant_clt_mtg}, 
	\begin{align}\label{ineq:be_clt_mu_aT_1}
	\mathfrak{d}_{ \mathrm{Kol} }\bigg( \bigg\{\frac{n_{a;T}}{\E n_{a;T}}\bigg\}^{\frac{1}{2}} Z_T, \mathsf{Z} \bigg)\lesssim_q \E^{1/3} \biggabs{\frac{n_{a;T}}{\E n_{a;T}}-1} +\frac{1}{ \{\E n_{a;T}\}^{\frac{q-1}{2q+1} }  }.
	\end{align}
	On the other hand, for any $t\in \R$ and $\rho \in (0,1/2)$, 
	\begin{align*}
	\Prob\big(Z_T\leq t\big)&\leq \Prob\bigg( \bigg\{\frac{n_{a;T}}{\E n_{a;T}}\bigg\}^{1/2} Z_T\leq t (1-\rho)^{-1} \bigg)+\Delta_{a;T}(\rho).
	\end{align*}
	So for any $t \in \R$ and $\rho \in (0,1/2)$,
	\begin{align*}
	&\Prob\big(Z_T\leq t\big)-\Prob\big(\mathsf{Z}\leq t\big)\\
	&\leq \Prob\bigg( \bigg\{\frac{n_{a;T}}{\E n_{a;T}}\bigg\}^{1/2} Z_T \leq t (1-\rho)^{-1} \bigg)-\Prob\big(\mathsf{Z}\leq t\big)+\Delta_{a;T}(\rho)\\
	&\leq \mathfrak{d}_{ \mathrm{Kol} }\bigg( \bigg\{\frac{n_{a;T}}{\E n_{a;T}}\bigg\}^{1/2} Z_T, \mathsf{Z} \bigg)+2\rho \abs{t}\cdot \sup_{u \in [t/2,2t]} \varphi(u)+ \Delta_{a;T}(\rho).
	\end{align*}
	Here $\varphi$ is the normal p.d.f. As a lower bound of the same form may be proved in a completely similar fashion, the claim (\ref{ineq:be_clt_mu_aT_step1}) follows by using the estimate (\ref{ineq:be_clt_mu_aT_1}) and the fact that $\sup_{t \in \R} \big\{\abs{t}\cdot \sup_{u \in [t/2,2t]} \varphi(u)\big\}<\infty$.
	
	\noindent (\textbf{Step 2}). We shall estimate in this step the three terms, denoted $I_1,I_2,I_3$, on the right hand side of (\ref{ineq:be_clt_mu_aT_step1}). Let us write $\err_\triangle(\Theta)\equiv D_\ast^{-1}\err(\Theta)+\vartheta_T^\ast$ and $\err_\ast(\Theta,\gamma)\equiv D_\ast^{-1}\err(\Theta)+\gamma/\gamma_T$. 
	
	For $I_1$, by using Proposition \ref{prop:E_n_aT}, for some universal constant $C_0>0$, uniformly in all $a \in [K]$,
	\begin{align}\label{ineq:be_clt_mu_aT_step2_0}
	\big(1-C_0 \cdot\err_\triangle(\Theta)\big)_+\cdot n_{a;T}^\ast\leq \E n_{a;T}\leq \big(1+C_0 \cdot\err_\triangle(\Theta)\big)\cdot n_{a;T}^\ast.
	\end{align}
	So for $\err_\triangle(\Theta)\leq 1/(2C_0)$, we have
	\begin{align}\label{ineq:be_clt_mu_aT_step2_I1}
	I_1\leq \bigg[\frac{1}{2 n_{a;T}^\ast}\big(\E\abs{n_{a;T}-n_{a;T}^\ast}-\abs{\E n_{a;T}-n_{a;T}^\ast}\big)\bigg]^{1/3}\leq C\cdot \err_\triangle^{1/3}(\Theta).
	\end{align}
	For $I_2$, using the first inequality of (\ref{ineq:be_clt_mu_aT_step2_0}), for $\err_\triangle(\Theta)\leq 1/(2C_0)$, we have
	\begin{align}\label{ineq:be_clt_mu_aT_step2_I2}
	I_2 \lesssim \bigg(\frac{1}{n_{a;T}^\ast}\bigg)^{1/2-\theta}\lesssim  \bigg(\frac{K}{T}+ \frac{\Delta_a^2}{\gamma_T^2}\bigg)^{1/2-\theta}\leq  \err_\triangle^{1/2-\theta}(\Theta).
	\end{align}
	For $I_3$, first using Proposition \ref{prop:E_n_aT}, 
	\begin{align*}
	\frac{\E n_{a;T}}{n_{a;T}}= \frac{ n_{a;T}^\ast }{n_{a;T}}\cdot\big(1+\bigo(\err_\triangle(\Theta))\big).
	\end{align*}
	On the other hand, by Proposition \ref{prop:sub_n_aT}, on an event $E_\ast(\gamma)$ with $\Prob\big(E_\ast^c(\gamma)\big)\leq C(\gamma_T^{-100}+\log T\cdot \gamma e^{-\gamma^2/2})$, by possibly adjusting $C_0>0$, 
	\begin{align*}
	\biggabs{\frac{ n_{a;T}^\ast }{n_{a;T}}-1} = \biggabs{1- \frac{ n_{a;T} }{n_{a;T}^\ast} }\cdot \frac{ n_{a;T}^\ast }{n_{a;T}} \leq C_0\cdot \err_\ast(\Theta,\gamma)\cdot \frac{ n_{a;T}^\ast }{n_{a;T}}.
	\end{align*}
	So for $\err_\ast(\Theta,\gamma)\leq 1/(2C_0)$, on the event $E_\ast(\gamma)$ we have $\abs{n_{a;T}^\ast/n_{a;T}-1}\lesssim \err_\ast(\Theta,\gamma)$. This means that on the same event,
	\begin{align*}
	\biggabs{\frac{\E n_{a;T}}{n_{a;T}}-1}&\leq C_1\cdot \big( \err_\ast(\Theta,\gamma)+\err_\triangle(\Theta)\big),
	\end{align*}
	provided $\err_\ast(\Theta,\gamma)\vee\err_\triangle(\Theta)< \min\{1/(2C_0),1/(4C_1)\}$. Consequently,
	\begin{align*}
	I_3\leq C_1\cdot \big( \err_\ast(\Theta,\gamma)+\err_\triangle(\Theta)\big)+ \Prob\big(E_\ast(\gamma)^c\big).
	\end{align*}
	By choosing $\gamma = C_2\sqrt{\log\log T}$ for a large enough constant $C_2>0$, 
	\begin{align}\label{ineq:be_clt_mu_aT_step2_I3}
	I_3&\lesssim\err_\triangle(\Theta)+ {\gamma_T^{-100}}+{ (\log T)^{-101}} \lesssim \err_\triangle(\Theta).
	\end{align}
	The claim for small values of $\err_\triangle(\Theta)$ now follows by combining (\ref{ineq:be_clt_mu_aT_step2_I1})-(\ref{ineq:be_clt_mu_aT_step2_I3}) with (\ref{ineq:be_clt_mu_aT_step1}). The case for large values of $\err_\triangle(\Theta)$ follows trivially by adjusting the numerical constant. \qed

\subsection{Proof of Theorem \ref{thm:ridge_CLT}}
	We assume $\sigma=1$ and shall drop the superscript `$\mathrm{lin}$' in $\Delta_a^{\mathrm{lin}}$, $\Theta^{\mathrm{lin}}$, $n_\ast^{\mathrm{lin}}$ and $n_{a;T}^{\ast,\mathrm{lin}}$ for notational simplicity. First, as $\hat{\beta}_\lambda=(S_T+\lambda I)^{-1} S_T \beta_\ast+T^{-1}(S_T+\lambda I)^{-1} X^\top \xi$, 
	\begin{align}\label{ineq:ridge_CLT_0}
	Z_T\equiv \sqrt{T}\cdot (S_T+\lambda I)\big(\hat{\beta}_\lambda-(S_T+\lambda I)^{-1} S_T \beta_\ast\big)= T^{-1/2}\cdot X^\top \xi. 
	\end{align}
	
	\noindent (\textbf{Step 1}). In this step, we show that if $\E S_T$ is non-singular, then there exists some universal constant $C>0$ such that for all $w \in \partial B_d(1)$ with $\max_{a \in [K]}\iprod{w}{ (\E S_T)^{-1/2} z_a}^2>0$,
	\begin{align}\label{ineq:be_clt_ridge_step1}
	\mathfrak{d}_{ \mathrm{Kol} }\big( \bigiprod{w}{(\E S_T)^{-1/2} Z_T}, \mathsf{Z}\big)\leq C\cdot \err_\triangle^{1/3}(\Theta).
	\end{align}
	To prove (\ref{ineq:be_clt_ridge_step1}), let us fix $v \in \R^d$ such that $\max_{a \in [K]}\iprod{v}{z_a}^2>0$. Let $Y_t\equiv T^{-1/2}\iprod{v}{x_t}\xi_t = \iprod{v}{Z_T}$. For $\alpha \in [K]$, let $\tau_a\equiv \iprod{v}{ z_a }^2$. Moreover, for $q\geq 1$, let $A^{(q)}$ be a random variable on $[K]$ with $\Prob(A^{(q)}=a)\propto \tau_a^q$ that is independent of all other variables. As $\pnorm{\tau}{\infty}>0$ by assumption, the law of $A^{(q)}$ is well-defined. For notational simplicity, we write $A\equiv A^{(1)}$. Using these notation, we have
	\begin{align*}
	s_T^2\equiv \sum_{t \in [T]} \E Y_t^2= \frac{1}{T}\sum_{a \in [K]} \tau_a \cdot \E n_{a;T} = \frac{1}{T}\cdot \E n_{A;T} = \pnorm{ (\E S_T)^{1/2} v }{}^2.
	\end{align*}
	A similar calculation leads to 
	\begin{align*}
	V_T^2 
	= \frac{ \sum_{a \in [K]} \tau_a\cdot n_{a;T} }{\sum_{a \in [K]} \tau_a\cdot \E n_{a;T} } = \frac{\E_A n_{A;T}}{\E n_{A;T} }.
	\end{align*}
	Using (\ref{ineq:be_clt_mu_aT_step2_0}), for $\err_\triangle(\Theta)\leq 1/(2C_0)$ (where $C_0$ is the constant in (\ref{ineq:be_clt_mu_aT_step2_0})),
	\begin{align}\label{ineq:be_clt_ridge_step1_1}
	\E \abs{V_T^2-1}&\lesssim \E_\xi \bigg[\frac{ \E_A \abs{n_{A;T}-  n_{A;T}^\ast}+ \E_A \abs{n_{A;T}^\ast- \E_\xi n_{A;T}} }{ \E_A n_{A;T}^\ast}\bigg]\lesssim \err_\triangle(\Theta). 
	\end{align}
	Here in the last inequality we used $\min_{a \in [K]} n_{a;T}^\ast \gtrsim 1$. On the other hand, for any $q\geq 1$, again with $M_q\equiv \E \mathsf{Z}^{2q}$, 
	\begin{align*}
	\frac{1}{s_T^{2q}}\sum_{t \in [T]} \pnorm{Y_t}{2q}^{2q}&=M_q \frac{ \sum_{a \in [K]} \tau_a^{q} \E n_{a;T} }{\big(\sum_{a \in [K]} \tau_a \E n_{a;T}\big)^{q}}\leq M_q \frac{  \E n_{A^{(q)};T} }{\E_{A^{(q)}} (\E_\xi n_{A^{(q)};T})^q }\leq \frac{M_q}{ \pnorm{\E_\xi n_{A^{(q)};T}}{q}^{q-1}   }.
	\end{align*}
	Consequently, for $\err_\triangle(\Theta)\leq 1/(2C_0)$, using the alternative estimate $\min_{a \in [K]} n_{a;T}^\ast \gtrsim \min\{T/K, \gamma_T^2/\pnorm{\Delta}{\infty}^2\}$, we have
	\begin{align}\label{ineq:be_clt_ridge_step1_2}
	\bigg(\frac{1}{s_T^{2q}}\sum_{t \in [T]} \pnorm{Y_t}{2q}^{2q}\bigg)^{ \frac{1}{2q+1} }\lesssim \pnorm{ n_{A^{(q)};T}^\ast }{q}^{  -\frac{q-1}{2q+1}  } \lesssim \bigg(\frac{K}{T}+\frac{ \pnorm{\Delta}{\infty}^2 }{\gamma_T^2}\bigg)^{ \frac{q-1}{2q+1} }. 
	\end{align}
	Now applying Proposition \ref{prop:quant_clt_mtg} with the estimates (\ref{ineq:be_clt_ridge_step1_1})-(\ref{ineq:be_clt_ridge_step1_2}) and $q$ chosen sufficiently large, for $\err_\triangle(\Theta)\leq 1/(2C_0)$ we have for all $v \in \R^d$ such that $\max_{a \in [K]}\iprod{v}{z_a}^2>0$,
	\begin{align*}
	\mathfrak{d}_{ \mathrm{Kol} }\bigg( \frac{ \iprod{v}{Z_T}}{\pnorm{ (\E S_T)^{1/2} v }{} }, \mathsf{Z}\bigg)\leq C\cdot \err_\triangle^{1/3}(\Theta).
	\end{align*}
	The desired claim (\ref{ineq:be_clt_ridge_step1}) now follows by choosing $v=(\E S_T)^{-1/2} w$ for $w \in \partial B_d(1)$ with $\max_{a \in [K]}\iprod{w}{ (\E S_T)^{-1/2} z_a}^2>0$ in the above display.

	\noindent (\textbf{Step 2}).  We will show in this step that there exists a large enough universal constant $C_0>0$ such that for the choice $\gamma\equiv C_0 \sqrt{\log\log T}$, if $\err_\triangle(\Theta)\leq \sigma_T^\ast/(2C_0^2 \mathsf{z}_K)$, on an event $E_\ast(\gamma)$ with $\Prob\big(E_\ast^c(\gamma)\big)\leq C(\gamma_T^{-100}+\log T\cdot \gamma e^{-\gamma^2/2})$, we have 
	\begin{align}\label{ineq:be_clt_ridge_step2}
	\pnorm{S_T^{-1/2}- (\E S_T)^{-1/2}}{\op} \leq 2C_0\cdot \mathsf{z}_K^{1/2} \err_\triangle^{1/2}(\Theta)/\sigma_T^{\ast}.
	\end{align}
	To this end, let
	\begin{align*}
	\zeta_T\equiv \max_{a \in [K]} \{\abs{\E n_{a;T}/n_{a;T}^\ast-1}+\abs{n_{a;T}/n_{a;T}^\ast-1}\}.
	\end{align*}
	We shall first prove that if $\mathsf{z}_K\cdot \zeta_T/\sigma_T^\ast\leq 1/2$, 
	\begin{align}\label{ineq:be_clt_ridge_step2_1}
	\pnorm{S_T^{-1/2}- (\E S_T)^{-1/2}}{\op} \leq 2 \big(\mathsf{z}_K \zeta_T/\sigma_T^{\ast,2}\big)^{1/2}.
	\end{align}
	To prove (\ref{ineq:be_clt_ridge_step2_1}), we need a few preliminary estimates:
	\begin{enumerate}
		\item $\pnorm{ S_T- S_T^{\ast} }{\op}\vee \pnorm{ \E S_T- S_T^{\ast} }{\op}\vee \pnorm{ S_T- \E S_T }{\op}\leq \mathsf{z}_K\cdot \zeta_T$.
		\item $\pnorm{S_T^{-1}}{\op}$: Note that
		\begin{align*}
		\pnorm{S_T^{-1}}{\op}&\leq \sigma_T^{\ast,-1}+ \pnorm{ S_T^{-1}- S_T^{\ast,-1} }{\op}\\
		&\leq \sigma_T^{\ast,-1}+ \sigma_T^{\ast,-1} \pnorm{ S_T- S_T^{\ast} }{\op}\cdot \pnorm{S_T^{-1}}{\op}.
		\end{align*}
		This means $\pnorm{S_T^{-1}}{\op}\leq 2/\sigma_T^\ast$, provided $\mathsf{z}_K\cdot \zeta_T/\sigma_T^\ast\leq 1/2$.
		\item $\pnorm{ (\E S_T)^{-1} }{\op}$: Using a completely similar argument as above, 
		\begin{align*}
		\pnorm{ (\E S_T)^{-1} }{\op}&\leq \sigma_T^{\ast,-1}+ \sigma_T^{\ast,-1} \pnorm{ \E S_T- S_T^{\ast} }{\op}\cdot \pnorm{(\E S_T)^{-1}}{\op}.
		\end{align*}
		This means again $\pnorm{(\E S_T)^{-1}}{\op}\leq 2/\sigma_T^\ast$, provided $\mathsf{z}_K\cdot \zeta_T/\sigma_T^\ast\leq 1/2$.
	\end{enumerate}
	Using the above preliminary estimates, if $\mathsf{z}_K\cdot \zeta_T/\sigma_T^\ast\leq 1/2$, 
	\begin{align*}
	&\pnorm{S_T^{-1/2}- (\E S_T)^{-1/2}}{\op}\\
	&\leq \pnorm{S_T^{-1/2}}{\op} \pnorm{ (\E S_T)^{-1/2} }{\op}\cdot \pnorm{S_T^{1/2}-(\E S_T)^{1/2}}{\op}\\
	&\stackrel{(\ast)}{\leq} (2/\sigma_T^\ast)\cdot \pnorm{S_T- \E S_T}{\op}^{1/2} \leq 2 \big(\mathsf{z}_K \zeta_T/\sigma_T^{\ast,2}\big)^{1/2}.
	\end{align*}
	Here in $(\ast)$ we used the fact that for two p.s.d. matrices $A,B$, $\pnorm{A^{1/2}-B^{1/2}}{\op}\leq \pnorm{A-B}{\op}^{1/2}$ (see, e.g., \cite[Proposition V.1.8 in combination with Theorem X.1.1]{bhatia1997matrix}). This completes the proof for (\ref{ineq:be_clt_ridge_step2_1}).
	
	Now using Proposition \ref{prop:sub_n_aT}, by choosing $\gamma\equiv C_0 \sqrt{\log\log T}$ for a large enough universal constant $C_0>0$, we have $\zeta_T\leq C_0^2 \err_\triangle(\Theta)$. This means that if $\err_\triangle(\Theta)\leq \sigma_T^\ast/(2C_0^2 \mathsf{z}_K)$, on the event $E_\ast(\gamma)$, we have $\pnorm{S_T^{-1/2}- (\E S_T)^{-1/2}}{\op} \leq 2C_0 \big(\mathsf{z}_K \err_\triangle(\Theta)/\sigma_T^{\ast,2}\big)^{1/2}$, proving (\ref{ineq:be_clt_ridge_step2}).
	
	\noindent (\textbf{Step 3}). In this step, we shall use the estimates in Step 1 and Step 2 to prove the claim in the theorem. First, using the martingale structure in $Z_T=T^{-1/2}\sum_{t \in [T]} x_t \xi_t$, we may compute
	\begin{align*}
	\E \pnorm{Z_T}{}^2 = \frac{1}{T} \sum_{t,s \in [T]} \E \xi_t \xi_s \iprod{x_t}{x_s} = \frac{1}{T}\sum_{t \in [T]} \E \pnorm{x_t}{}^2 = \mathsf{z}_K+\bigo\big(\err_\triangle(\Theta)\big). 
	\end{align*}
	Consequently, for any $t \in \R$, $\epsilon>0$ and $u>0$, if $\err_\triangle(\Theta)\leq \sigma_T^\ast/(2C_0^2 \mathsf{z}_K)$,
	\begin{align*}
	&\Prob\big(\iprod{w}{S_T^{-1/2} Z_T}\leq t \big)-\Prob\big(\mathsf{Z}\leq t\big)\\
	&\leq \Prob\big(\iprod{w}{(\E S_T)^{-1/2} Z_T}\leq t +\epsilon\big)-\Prob\big(\mathsf{Z}\leq t\big)\\
	&\qquad +\Prob\Big( \pnorm{S_T^{-1/2}-(\E S_T)^{-1/2}}{\op}\cdot \big(\mathsf{z}_K+C\cdot  \err_\triangle(\Theta)\big)^{1/2}\cdot u \geq \epsilon \Big)+ u^{-2}\\
	&\leq \mathfrak{d}_{ \mathrm{Kol} }\Big( \bigiprod{w}{(\E S_T)^{-1/2} Z_T}, \mathsf{Z}\Big) + \epsilon + u^{-2}\\
	&\qquad + \bm{1}\Big(\epsilon \leq 2C_0\cdot \Big\{\mathsf{z}_K+C\cdot \err_\triangle^{1/2}(\Theta) \mathsf{z}_K^{1/2}\Big\}\cdot u\cdot  \err_\triangle^{1/2}(\Theta)/\sigma_T^{\ast} \Big)+\Prob\big(E_\ast(\gamma)^c\big).
	\end{align*}
	Here in the last inequality we used the estimate (\ref{ineq:be_clt_ridge_step2}) in Step 2. Using the estimate (\ref{ineq:be_clt_ridge_step1}) in Step 1, if $\err_\triangle(\Theta)\leq \sigma_T^\ast/(2C_0^2 \mathsf{z}_K)$, it holds for any $t \in \R$ and $u>0$ that
	\begin{align*}
	&\Prob\big(\iprod{w}{S_T^{-1/2} Z_T}\leq t \big)-\Prob\big(\mathsf{Z}\leq t\big)\\
	&\leq C\cdot \err_\triangle^{1/3}(\Theta)+ C\cdot  \big(\mathsf{z}_K/\sigma_T^\ast+1\big)\cdot \err_\triangle^{1/2}(\Theta)\cdot u+ u^{-2}.
	\end{align*}
	Optimizing over $u>0$, for some universal constant $C_1>0$, if $\err_\triangle(\Theta)\leq \sigma_T^\ast/(C_1 \mathsf{z}_K)$, it holds for any $t \in \R$ that
	\begin{align*}
	&\Prob\big(\iprod{w}{S_T^{-1/2} Z_T}\leq t \big)-\Prob\big(\mathsf{Z}\leq t\big)\\
	&\leq C\cdot \err_\triangle^{1/3}(\Theta)\cdot \big(1+ \big(\mathsf{z}_K/\sigma_T^\ast+1\big)^{2/3} \err_\triangle^{1/3}(\Theta)\big)\\
	&\leq C_1 \cdot \err_\triangle^{1/3}(\Theta)\cdot \big(1+ \big(\mathsf{z}_K/\sigma_T^\ast\big)^{1/3}\big).
	\end{align*}
	A lower bound can be proved similarly, so for all $w \in \partial B_d(1)$ with $\max_{a \in [K]}\iprod{w}{ (\E S_T)^{-1/2} z_a}^2>0$,
	\begin{align*}
	\mathfrak{d}_{ \mathrm{Kol} }\big( \bigiprod{w}{ S_T^{-1/2} Z_T}, \mathsf{Z}\big)\leq C_1 \cdot \err_\triangle^{1/3}(\Theta)\cdot \big(1+ \big(\mathsf{z}_K/\sigma_T^\ast\big)^{1/3}\big),
	\end{align*}
	provided the condition $\err_\triangle(\Theta)\leq \sigma_T^\ast/(C_1 \mathsf{z}_K)$ holds. Moreover, under this condition, $\E S_T$ is invertible, and therefore $\{(\E S_T)^{-1/2}z_a: a\in [K]\}$ spans $\R^d$. This means that the constraint  $\max_{a \in [K]}\iprod{w}{ (\E S_T)^{-1/2} z_a}^2>0$ can be dropped for free under $\err_\triangle(\Theta)\leq \sigma_T^\ast/(C_1 \mathsf{z}_K)$.  If $\err_\triangle(\Theta)> \sigma_T^\ast/(C_1 \mathsf{z}_K)$, the above estimate holds trivially by adjusting the numerical constant. \qed

\appendix

\section{A boundary crossing probability estimate}\label{section:appendix_BM_crossing}

\begin{lemma}\label{lem:BM_crossing}
	Let $\xi_1,\ldots,\xi_T$ be i.i.d. $\mathcal{N}(0,1)$. Then there exists a universal constant $C>0$ such that the following hold for $x\geq 1,y>0$:
	\begin{enumerate}
		\item $
		\mathfrak{p}_T(x)\equiv \Prob\big(\max_{t \in [T]} t^{-1/2}\sum_{s \in [t]} \xi_s>x \big)\leq C \cdot \log T\cdot  x e^{-x^2/2}$. 
		\item $
		\mathfrak{q}_T(x,y)\equiv \Prob\big(\min_{t \in [T]}\big\{ t^{-1}\sum_{s \in [t]} \xi_s+t^{-1/2}x  \big\}\leq -y\big)\leq C y^{-2} \int_x^\infty (z-x)^2 e^{-z^2/2}\,\d{z}$. 
	\end{enumerate} 
\end{lemma}

We remark that the exact asymptotic expression of the boundary crossing probability in Lemma \ref{lem:BM_crossing}-(1) may be computed by the method of Wald's likelihood ratio, cf. \cite{siegmund1985sequential,siegmund1986boundary}. 
Non-asymptotic upper bounds for such boundary crossing probabilities can be also obtained 
through the same method; see \cite{ren2024lai} where it is proved 
$\mathfrak{p}_T(x)\leq \varphi(x)\{x(\log T+2)+\sqrt{2/\pi}+4/x\}$. 
Here we obtain a non-asymptotic estimate of the correct order via a fairly elementary blocking method when $\{\xi_i\}$'s are Gaussian. 

\begin{proof}[Proof of Lemma \ref{lem:BM_crossing}]
	For any two positive real numbers $a,b$, in the proof below we shall use $\sum_{t \in [a,b]}$ as $\sum_{t \in [a,b]\cap \mathbb{Z}}$.
	
	\noindent (1). 
     Using a standard blocking argument, we have for any $\eta \in (1,2]$,
	\begin{align*}
	\mathfrak{p}_T(x) &\leq \sum_{1\leq r \leq \ceil{\log_\eta T} } \Prob\bigg(\max_{\eta^{r-1}\leq t \leq \eta^r} \frac{1}{\sqrt{t}}\sum_{s \in [t]} \xi_s>x \bigg) \leq \sum_{1\leq r \leq \ceil{\log_\eta T} } \Prob\bigg(\max_{ t \in [\eta^r]} \sum_{s \in [t]} \xi_s> \sqrt{\eta^{r-1}} x \bigg).
	\end{align*}
	Apply L\'evy's maximal inequality (cf. \cite[Theoren 1.1.1]{de2012decoupling}), we have 
	\begin{align*}
	\mathfrak{p}_T(x)&\leq 2\sum_{1\leq r \leq \ceil{\log_\eta T} } \Prob\bigg(\sum_{s \in [\eta^r]} \xi_s> \sqrt{\eta^{r-1}} x \bigg)\leq C\cdot \bigg(1+\frac{\log T}{\log \eta}\bigg)\cdot x^{-1}\cdot  e^{-x^2/(2\eta)}.
	\end{align*}
	By writing $\eta^{-1}=1-\epsilon$ for some $\epsilon \in (0,10^{-1})$, the right hand side of the above display can be further bounded by an absolute constant multiple of $  \log T\cdot \epsilon^{-1} x^{-1} e^{-(1-\epsilon)x^2/2}$, and therefore
	\begin{align*}
	\mathfrak{p}_T(x) \leq C\cdot \log T\cdot x^{-1}\cdot  \inf_{\epsilon \in (0,10^{-1})} \epsilon^{-1}\cdot e^{-(1-\epsilon)x^2/2}.
	\end{align*}
	Now for $x> 4$, by choosing $\epsilon=1/x^2 \in (0,10^{-1})$, we have $(1-\epsilon)x^2=x^2-1$ and therefore the right hand side of the above display is bounded by $C\cdot \log T\cdot x e^{-x^2/2}$. By adjusting the numerical constant $C$, we may include the case $1\leq x\leq 4$ as well.
	
	\noindent (2). The claim is proved in \cite[Lemma 8]{ren2024lai}. 
\end{proof}

\section*{Acknowledgments}

Q. Han would like to thank K. Kato for several helpful conversations and pointers on quantitative martingale central limit theorems.
 
The research of Q. Han is partially supported by NSF grant DMS-2143468. The research of K. Khamaru is partially supported by NSF grant DMS-2311304. The research of C.-H. Zhang is partially supported by NSF grants DMS-2052949 and DMS-2210850.

\bibliographystyle{amsalpha}
\bibliography{mybib}

\end{document}